\documentclass[12pt]{amsart}
\usepackage{enumerate,amssymb,version,aliascnt,geometry,stmaryrd,mathrsfs,chngcntr,hyperref,calrsfs,mathtools}
\usepackage[all]{xy}
\usepackage[mathscr]{euscript}
\hypersetup{
pdfauthor={Olivier Haution},
pdftitle={},
pdfkeywords={},
hidelinks
}

\usepackage[T1]{fontenc}
\usepackage{textcomp}

\let\origsetminus\setminus
\let\originfty\infty
\let\origpartial\partial
\let\origin\in
\let\origprod\prod
\let\origsum\sum
\let\origsubset\subset
\let\origsimeq\simeq
\let\origast\ast
\let\origsum\sum
\usepackage{mathabx}
\let\setminus\origsetminus
\let\infty\originfty
\let\partial\origpartial
\let\in\origin
\let\prod\origprod
\let\subset\origsubset
\let\simeq\origsimeq
\let\ast\origast
\let\sum\origsum

\DeclareMathOperator{\Hh}{H}
\DeclareMathOperator{\GL}{GL}
\DeclareMathOperator{\SL}{SL}
\DeclareMathOperator{\Spec}{Spec}
\DeclareMathOperator{\Sp}{Sp}
\DeclareMathOperator{\Th}{Th}
\DeclareMathOperator{\SH}{SH}

\DeclareMathOperator{\id}{id}
\DeclareMathOperator{\Hom}{Hom}
\DeclareMathOperator{\Sm}{Sm}
\DeclareMathOperator{\Sch}{Sch}
\DeclareMathOperator{\Schq}{Schq}
\DeclareMathOperator{\KGL}{KGL}
\DeclareMathOperator{\Ze}{Z}
\DeclareMathOperator{\Zei}{I}

\DeclareMathOperator{\Mor}{Mor}
\DeclareMathOperator{\colim}{colim}
\DeclareMathOperator{\Sph}{S}

\newcommand{\Qcoh}{\mathsf{Qcoh}}
\newcommand{\Ac}{\mathcal{A}}

\newcommand{\Vc}{\mathcal{V}}
\newcommand{\St}{\mathcal{P}}
\newcommand{\Eu}{\mathcal{E}}
\newcommand{\eu}{E}
\newcommand{\Fc}{\mathcal{F}}
\newcommand{\Uns}{\mathcal{Un}}
\newcommand{\Cc}{\mathcal{C}}

\newcommand{\Un}{\mathbf{1}}
\newcommand{\Su}{\Sigma^{\infty}}
\newcommand{\Sup}{\Sigma^{\infty}_+}
\newcommand{\Oc}{\mathcal{O}}
\newcommand{\Ec}{\mathcal{E}}
\newcommand{\Gc}{\mathcal{G}}
\newcommand{\Zz}{\mathbb{Z}}
\newcommand{\Pp}{\mathbb{P}}
\newcommand{\Nn}{\mathbb{N}}
\newcommand{\Tan}{T}

\newcommand{\Fp}{\mathbb{F}_p}

\newtheorem*{theorem*}{Theorem}
\newtheorem*{proposition*}{Proposition}
\newtheorem{thm}{Theorem}
\newtheorem{prop}[thm]{Proposition}
\newtheorem{cor}[thm]{Corollary}

\swapnumbers

\newtheorem{theorem}{Theorem}[section]
\newaliascnt{proposition}{theorem}
\newtheorem{proposition}[proposition]{Proposition}
\newaliascnt{lemma}{theorem}
\newtheorem{lemma}[lemma]{Lemma}
\newaliascnt{corollary}{theorem}
\newtheorem{corollary}[corollary]{Corollary}

\theoremstyle{definition}
\newaliascnt{remark}{theorem}
\newtheorem{remark}[theorem]{Remark}
\newaliascnt{example}{theorem}
\newtheorem{example}[example]{Example}
\newaliascnt{definition}{theorem}
\newtheorem{definition}[definition]{Definition}
\newaliascnt{notation}{theorem}

\newtheoremstyle{par}%                % Name
  {}%                                     % Space above
  {}%                                     % Space below
  {}%                                     % Body font
  {}%                                     % Indent amount
  {}%                 	                  % Theorem head font
  {.}%                                    % Punctuation after theorem head
  { }%                                    % Space after theorem head, ' ', or \newline
  {}%
\theoremstyle{par}
\newtheorem{para}[theorem]{}

\counterwithin{equation}{theorem}
\numberwithin{equation}{theorem}

\newcommand{\rref}[1]{(\ref{#1})}
\newcommand{\dref}[2]{(\ref{#1}.\ref{#2})}

\makeatletter
\newcommand*{\doublerightarrow}[2]{\mathrel{
  \settowidth{\@tempdima}{$\scriptstyle#1$}
  \settowidth{\@tempdimb}{$\scriptstyle#2$}
  \ifdim\@tempdimb>\@tempdima \@tempdima=\@tempdimb\fi
  \mathop{\vcenter{
    \offinterlineskip\ialign{\hbox to\dimexpr\@tempdima+1em{##}\cr
    \rightarrowfill\cr\noalign{\kern.5ex}
    \rightarrowfill\cr}}}\limits^{\!#1}_{\!#2}}}
\makeatother

\begin{document}
\begin{abstract}
We establish a purely geometric form of the concentration theorem (also called localization theorem) for actions of a linearly reductive group $G$ on an affine scheme $X$ over an affine base scheme $S$. It asserts the existence of a $G$-representation without trivial summand over $S$, which acquires over $X$ an equivariant section vanishing precisely at the fixed locus of $X$.

As a consequence, we show that the equivariant stable motivic homotopy theory of a scheme with an action of a linearly reductive group is equivalent to that of the fixed locus, upon inverting appropriate maps, namely the Euler classes of representations without trivial summands. We also discuss consequences for equivariant cohomology theories obtained using Borel's construction. This recovers most known forms of the concentration theorem in algebraic geometry, and yields generalizations valid beyond the setting of actions of diagonalizable groups on one hand, and that of oriented cohomology theories on the other hand. 

Finally, we derive a version of Smith theory for motivic cohomology, following the approach of Dwyer--Wilkerson in topology.
\end{abstract}

\author{Olivier Haution}
\title{The geometric concentration theorem}
\email{olivier.haution at gmail.com}
\address{Mathematisches Institut, Ludwig-Maximilians-Universit\"at M\"unchen, Theresienstr.\ 39, D-80333 M\"unchen, Germany}
\address{Dipartimento di Matematica e Applicazioni, Università degli Studi di Milano-Bicocca, via Roberto Cozzi 55, 20125 Milano, Italy}
\thanks{This work was supported by the DFG research grant HA 7702/5-1 and Heisenberg grant HA 7702/4-1.}

\subjclass[2010]{}

\date{\today}

\maketitle

\numberwithin{theorem}{section}
\numberwithin{lemma}{section}
\numberwithin{proposition}{section}
\numberwithin{corollary}{section}
\numberwithin{example}{section}
\numberwithin{notation}{section}
\numberwithin{definition}{section}
\numberwithin{remark}{section}

\section*{Introduction}
The localization theorem has been a very influential tool in algebraic topology (see e.g.\ \cite{Atiyah-Segal-Index2,Quillen-Spectrum-equiv}). Roughly speaking it asserts that a manifold with a group action has the same equivariant cohomology as its fixed locus, upon inverting certain elements. Various fixed-point theorems may be viewed as consequences of the localization theorem for appropriate cohomology theories.

There are of course analogs in algebraic geometry, starting with Thomason's ``concentration'' theorem for $K$-theory \cite{Thomason-Lefschetz}, followed later by Edidin--Graham's version for Chow groups \cite{EG-Equ}. More recently, Levine proved a version for Witt cohomology \cite{Levine-Atiyah-Bott}. Let us also mention the paper of Ravi--Khan \cite{Khan-Ravi-cohomological}, which appeared essentially at the same time as the current paper and whose results appear to partly overlap those obtained here, even though the methods seem quite different.\\

All previous approaches to the concentration theorem in algebraic geometry (known to the author) involve stratification techniques, via Thomason's generic slice theorem \cite{Thomason-comparison}. In the present paper we take a somewhat different path, starting from a purely geometric form of the concentration theorem. To state it, let us consider a group scheme $G$ which is affine and of finite type over a noetherian scheme $S$. We assume that:
\begin{enumerate}[(a)]
\item
\label{enum:int:1}
$G$ is linearly reductive,

\item
\label{enum:int:2}
$G$ has the resolution property,

\item
\label{enum:int:3}
the $\Oc_S$-module $\Oc_S[G]$ is locally projective.
\end{enumerate}
We think of the assumptions \eqref{enum:int:2} and \eqref{enum:int:3} as mostly technical and likely to be verified in many concrete cases (they are automatic when $S$ is the spectrum of field). The assumption \eqref{enum:int:1} is more serious, and cannot be entirely avoided (at least one should exclude unipotent groups, see \rref{rem:unipotent}).

Let us consider the set of $G$-representations without trivial summand, namely
\[
\Vc_G = \{\text{$G$-equivariant vector bundles $V \to S$ such that $V^G=0$}\}.
\]
Here is our most primitive form of the concentration theorem:
\begin{thm}[Geometric concentration](See \rref{th:conc_main}.)
\label{th:int_conc_geom}
Assume that $S$ is affine, and let $X$ be an affine noetherian $G$-equivariant $S$-scheme. Then some element $V \in \Vc_G$ admits, when pulled back to $X$, a $G$-equivariant section whose vanishing locus is the fixed locus $X^G$.
\end{thm}

Let us sketch how the more standard form of the concentration theorem for a cohomology theory $H$ is derived from this theorem. In the situation of the theorem assume additionally that $X^G=\varnothing$. If $H$ admits a theory of Euler classes, then the Euler class $e(V)$ will vanish in $H(X)$, and thus $H(X)[E_G^{-1}]=0$ where $E_G \subset H(S)$ is the set of Euler classes of the elements of $\Vc_G$. If $H$ has the homotopy invariance property, one may remove the assumption that $X$ is affine using Jouanolou's trick. Finally, if there is a localization sequence for the theory $H$, we deduce that $H(X)[E_G^{-1}] \simeq H(X^G)[E_G^{-1}]$ for any scheme $X$ with a $G$-action, which is the concentration theorem.\\

This argument can be made precise using the equivariant stable motivic homotopy category $\SH^G(-)$ defined by Hoyois \cite{Hoyois-equivariant}. Consider the set of ``Euler classes morphisms'' in $\SH^G(S)$
\[
\Eu_G = \{ e_V \colon \Un_S \to \Sph^V, \; \text{for $V \in \Vc_G$}\}.
\]
Here the sphere $\Sph^V=\Su \Th V$ is the infinite suspension of the Thom space of $V$, and $e_V$ is the cone of the map induced by the projection $V^\circ \to S$ from the complement of the zero-section in $V$.

\begin{thm}[Motivic concentration](See \rref{th:localisation}.)
\label{th:int_conc_motivic}
Let $X$ be a $G$-quasi-projective $S$-scheme. Then we have an adjoint equivalence $\SH^G(X)[\Eu_G^{-1}] \simeq \SH^G(X^G)[\Eu_G^{-1}]$.
\end{thm}
More concretely:
\begin{cor}
\label{cor_intro}
Let $X$ be a smooth $G$-quasi-projective $S$-scheme. Then the closed immersion $X^G \to X$ becomes an isomorphism in $\SH^G(S)[\Eu_G^{-1}]$
\end{cor}
Let us observe that, in $\SH^G(S)$, inverting the morphisms in $\Eu_G$ amounts to imposing the vanishing of the objects $\Sup (V^{\circ})$ for $V \in \Vc_G$. Since $V^\circ = V \smallsetminus V^G$ for such $V$, it is clearly necessary to do so in order to identify, in a quotient of $\SH^G(S)$, every smooth $G$-equivariant $S$-scheme with its fixed locus. Corollary~\ref{cor_intro} asserts that one needs impose no further vanishing beyond these evident ones.\\

Next, we discuss in \S\ref{sect:coh} the consequences for theories which are representable in $\SH^G(S)$. To a ring spectrum $A \in \SH^G(S)$ corresponds a cohomology theory $A^{*,*}$ and a Borel--Moore homology theory $A_{*,*}$. What we call a \emph{pseudo-orientation} of $A$ is the data of elements $\beta_E \in A^{*,*}(X;-E)=A^{*,*}(\Th E)$ for each $G$-equivariant vector bundle $E \to X$ with $X$ smooth over $S$, satisfying certain conditions (see \rref{def:beta:1}). This permits to define classes $b(E) = e_E^*(\beta_E) \in A^{*,*}(X)$ for such $E$, and to consider the set
\[
E_G = \{b(V) \text{ for $V \in \Vc_G$}\} \subset A^{*,*}(S).
\]
An example of a theory represented by a spectrum $A \in \SH^G(S)$ equipped with a pseudo-orientation is equivariant homotopy $K$-theory defined by Hoyois \cite{Hoyois-cdh-equiv}. (Further examples in the non-equivariant case $G=1$ will be given below.)

\begin{thm}[See \rref{prop:conc_equiv_coh} and \rref{prop:inverse_1}]
\label{th:int_conc_equiv}
Let $A \in \SH^G(S)$ be a ring spectrum equipped with a pseudo-orientation. Let $X$ be a $G$-quasi-projective $S$-scheme, and $i \colon X^G \to X$ the immersion of the fixed locus. Then we have isomorphisms
\[
i^* \colon A^{*,*}(X)[E_G^{-1}] \xrightarrow{\sim} A^{*,*}(X^G)[E_G^{-1}] \quad \text{and} \quad i_* \colon A_{*,*}(X^G)[E_G^{-1}] \xrightarrow{\sim} A_{*,*}(X)[E_G^{-1}].
\]
If in addition $X$ is smooth over $S$, and $N$ is the normal bundle to $i$, then
\[
i_*(b(N)^{-1} \beta_N) = 1 \in A^{*,*}(X)[E_G^{-1}].
\]
\end{thm}
In practice though, equivariant theories often arise in a somewhat different way. Indeed, every non-equivariant cohomology theory can be turned into an equivariant one using Borel's construction, a process which involves considering approximations of the space $EG$, and passing to the limit. In \S\ref{sect:Borel}, we prove concentration theorems for such theories. Some care is required due to interactions between limits and localization at a multiplicative subset, and at that point having at hand a geometric result such as Theorem~\ref{th:int_conc_geom} becomes particularly useful.

Explicitly let us consider a ring spectrum $A \in \SH(S)$, and assume that $A$ is equipped with a (non-equivariant) pseudo-orientation. For instance, in case $A$ is commutative, any classical orientation of $A$ (by the groups $\GL,\SL^c,\SL,\Sp$, or even a hyperbolic orientation in the sense of \cite{hyp}) yields a pseudo-orientation of $A$. When $A$ is $\GL$-oriented $b(E)$ is the Euler class of $E$, while it is its Pontryagin class for the other orientations mentioned above.

For a given family $E_mG$ (for $m \in \Nn$) of $G$-equivariant motivic spaces approximating of $EG$, Borel's construction yields theories $A^{*,*}_G(-)$ and $A_{*,*}^G(-)$ (see \rref{p:Borel}). Explicitly, when $X$ is a $G$-quasi-projective $S$-scheme, we have for $p,q \in \Zz$ 
\[
A^{p,q}_G(X) = \lim_m A^{p,q}((X \times_S E_mG)/G).
\]
From the pseudo-orientation of $A$, we obtain classes $\beta^G_E \in A^{*,*}_G(X;-E)$ and $b^G(E) \in A^{*,*}_G(E)$ for each  $G$-equivariant vector bundle $E \to X$ with $X$ smooth over $S$ (see \rref{p:beta}). Writing as above $E_G = \{b^G(V), V \in \Vc_G\}\subset A^{*,*}_G(S)$, we have the analog of Theorem~\ref{th:int_conc_equiv}:

\begin{thm}[See \rref{prop:conc_Borel-type} and \rref{prop:formula_Borel-type}]
\label{th:int_conc_Borel}
Let $X$ be a $G$-quasi-projective $S$-scheme, and $i \colon X^G \to X$ the immersion of the fixed locus. Then we have isomorphisms
\[
i^* \colon A^{*,*}_G(X)[E_G^{-1}] \xrightarrow{\sim} A^{*,*}_G(X^G)[E_G^{-1}] \quad \text{and} \quad i_* \colon A_{*,*}^G(X^G)[E_G^{-1}] \xrightarrow{\sim} A_{*,*}^G(X)[E_G^{-1}].
\]
If in addition $X$ is smooth over $S$, and $N$ is the normal bundle to $i$, then
\[
i_*(b^G(N)^{-1} \beta^G_N) = 1 \in A^{*,*}_G(X)[E_G^{-1}].
\]
\end{thm}

Next, we obtain the following result, which is perhaps disappointing:
\begin{prop}[see \rref{rem:mult_type_oriented}]
\label{prop:int_GL}
Let $S$ be the spectrum of a field, and $A \in \SH(S)$ a $\GL$-oriented commutative ring spectrum. If the group $G$ is not of multiplicative type, then $A^{*,*}_G(S)[E_G^{-1}]=0$.
\end{prop}
This implies that the practical applications of Theorem~\ref{th:int_conc_Borel} for $\GL$-oriented theories will be limited to groups $G$ which are of multiplicative type (most previously known forms of the concentration theorem are restricted to diagonalizable groups, or even tori). Note however that making this observation does require the general form of Theorem~\ref{th:int_conc_Borel} (the proof of Proposition~\ref{prop:int_GL} actually uses the validity of Theorem~\ref{th:int_conc_Borel} for $G$ arbitrary).

At the other extreme, when $A$ is an $\eta$-periodic theory with a pseudo-orientation, the localization $A^{*,*}_G(S)[E_G^{-1}]$ is nonzero only for those groups $G$ admitting no odd-rank vector bundle $V \to S$ such that $V^G=0$, a class which contains no nontrivial diagonalizable group (see \rref{rem:eta}). This makes the concrete applications beyond the $\GL$-oriented case quite limited at the moment. Let us mention that Levine's approach in \cite{Levine-Atiyah-Bott} does apply to such theories, but in return requires to make rather restrictive assumptions on the orbit types of the action.\\

We conclude the paper with an application of the concentration theorem to motivic cohomology with finite coefficients, obtaining in \S\ref{sect:Smith} an algebraic version of Dwyer--Wilkerson result in topology \cite{Dwyer-Wilkerson}. In brief, it expresses the non-equivariant cohomology of the fixed locus purely in terms of the equivariant cohomology of the ambient scheme together with its structure as a module over the Steenrod algebra. 

To be more precise, let $p$ be a prime number and assume that the base scheme $S$ is the spectrum of a perfect field of characteristic unequal to $p$. We denote by $\Hh^{*,*}$ the motivic cohomology with $\Fp$-coefficients, and by $\Hh_G^{*,*}$ the corresponding Borel-type equivariant theory. These are equipped with an action of the modulo $p$ motivic Steenrod algebra $\Ac$. When $N$ is an $\Ac$-module, we denote by $\Uns(N) \subset N$ the set of unstable elements (those elements on which operations of sufficiently high degree vanish). We prove:
\begin{thm}[see \rref{cor:Smith}]
Let $G=(\mu_p)^n$ for some $n \in \Nn$, and $X$ a smooth $G$-quasi-projective $S$-scheme. Then we have a natural isomorphism
\[
\Hh^{*,*}(X^G) \simeq \Hh^{*,*}(S) \otimes_{\Hh^{*,*}_G(S)} \Uns(\Hh^{*,*}_G(X)[\eu_G^{-1}]).
\]
\end{thm}

In the whole paper, instead of the closed subscheme $X^G \subset X$ we will usually consider the fixed locus $X^C$, where $C \subset G$ is a closed normal subgroup which is linearly reductive. This leads to slight generalizations of all the results stated above; in this introduction we have limited ourselves to the case $C=G$ for the sake of simplicity.\\

\noindent \textbf{Acknowledgments}.
I am grateful to the anonymous referee for the suggestions, which helped improve the presentation.

\section{Group schemes and their actions}
\numberwithin{theorem}{subsection}
\numberwithin{lemma}{subsection}
\numberwithin{proposition}{subsection}
\numberwithin{corollary}{subsection}
\numberwithin{example}{subsection}
\numberwithin{notation}{subsection}
\numberwithin{definition}{subsection}
\numberwithin{remark}{subsection}

In this paper, we work over a noetherian base scheme $S$. A group scheme over $S$ will mean a flat, finite type, affine group scheme over $S$. For an $S$-scheme $X$ we denote by $\pi_X \colon X \to S$ its structure morphism.

\subsection{The fixed subsheaf}
In this section, we let $G$ be a group scheme over $S$.
\begin{para}
\label{p:group}
The $S$-scheme $G$ is the relative spectrum of a unique quasi-coherent $\Oc_S$-algebra that we denote by $\Oc_S[G]$. The $S$-scheme structure yields morphisms of $\Oc_S$-algebras
\[
\Oc_S \xrightarrow{\iota} \Oc_S[G], \quad \Oc_S[G] \otimes_{\Oc_S} \Oc_S[G] \to \Oc_S[G],
\]
and the group structure yields morphisms of $\Oc_S$-algebras
\begin{equation}
\label{eq:Os_G_group}
\Oc_S[G] \xrightarrow{\varepsilon} \Oc_S, \quad \Oc_S[G] \xrightarrow{\chi_G} \Oc_S[G], \quad \Oc_S[G] \xrightarrow{\mu_G} \Oc_S[G]\otimes_{\Oc_S} \Oc_S[G],
\end{equation}
satisfying the axioms of a Hopf algebra. 
\end{para}

\begin{para}
\label{p:def_equiv_module}
In the notation of \rref{p:group}, a $G$-equivariant quasi-coherent $\Oc_S$-module is a quasi-coherent $\Oc_S$-module $\Fc$ together with a morphism of $\Oc_S$-modules
\[
\mu_{\Fc} \colon \Fc \to \Oc_S[G] \otimes_{\Oc_S} \Fc
\]
(the ``coaction morphism''), such that:
\begin{enumerate}[(i)]
\item 
\label{p:def_equiv_module:1}
the identity of $\Fc$ factors as
\[
\Fc \xrightarrow{\mu_{\Fc}}  \Oc_S[G] \otimes_{\Oc_S} \Fc \xrightarrow{\varepsilon \otimes \id_{\Fc}} \Oc_S \otimes_{\Oc_S} \Fc = \Fc,
\]

\item
\label{p:def_equiv_module:2}
the following composites coincide:
\[
\Fc \xrightarrow{\mu_{\Fc}} \Oc_S[G] \otimes_{\Oc_S} \Fc \doublerightarrow{(\id_{\Oc_S[G]}) \otimes \mu_{\Fc}}{\mu_G \otimes \id_{\Fc}}  \Oc_S[G] \otimes_{\Oc_S} \Oc_S[G] \otimes_{\Oc_S} \Fc.
\]
\end{enumerate}
A morphism of $G$-equivariant quasi-coherent $\Oc_S$-modules is a morphism of $\Oc_S$-modules compatible with the coaction morphisms. We thus obtain a category that we denote by $\Qcoh^G(S)$. The category of quasi-coherent $\Oc_S$-modules will be denoted by $\Qcoh(S)$.
\end{para}

\begin{para}
Let $\Fc \in \Qcoh^G(S)$, and $\Gc \subset \Fc$ a quasi-coherent $\Oc_S$-submodule. Since $G$ is flat over $S$, it follows that $\Gc$ admits at most one $G$-equivariant structure compatible with the inclusion $\Gc \subset \Fc$. If such a structure exists, we will say that the submodule $\Gc$ is \emph{$G$-invariant}.
\end{para}

\begin{para}
\label{p:tensor_equiv}
For $\Fc,\Gc \in \Qcoh^G(S)$, we will view the tensor product $\Fc \otimes_{\Oc_S} \Gc$ as an object of $\Qcoh^G(S)$, via the coaction morphism
\[
\Fc \otimes_{\Oc_S} \Gc \xrightarrow{\mu_\Fc \otimes \mu_\Gc} \Oc_S[G] \otimes_{\Oc_S} \Fc \otimes_{\Oc_S}  \Oc_S[G] \otimes_{\Oc_S} \Gc \to \Oc_S[G] \otimes_{\Oc_S} \Fc \otimes_{\Oc_S} \Gc,
\]
where the last morphism is given locally by $a \otimes b \otimes c \otimes d \mapsto ac\otimes b \otimes d$.
\end{para}

\begin{definition}
\label{def:fixed_subsheaf}
Let $\Fc$ be a $G$-equivariant quasi-coherent $\Oc_S$-module. The fixed subsheaf of $\Fc$ is defined as
\[
\Fc^G = \ker( \Fc \xrightarrow{1 \otimes \id_{\Fc} -\mu_{\Fc}} \Oc_S[G] \otimes_{\Oc_S} \Fc),
\]
where $1 \otimes \id_{\Fc}$ denotes the morphism  $\Fc = \Oc_S \otimes_{\Oc_S} \Fc \xrightarrow{\iota \otimes \id_\Fc} \Oc_S[G] \otimes_{\Oc_S} \Fc$. This permits to define a functor $(-)^G \colon \Qcoh^G(S) \to \Qcoh(S)$.
\end{definition}

\begin{para}
\label{p:fp_adjoint}
The functor $(-)^G \colon \Qcoh^G(S) \to \Qcoh(S)$ is right adjoint to the functor $\Qcoh(S) \to \Qcoh^G(S)$ which endows an $\Oc_S$-module with the trivial $G$-action.
\end{para}

\begin{lemma}
\label{lemm:inv_is_equ}
Let $C \subset G$ be a closed normal subgroup, and $\Fc \in \Qcoh^G(S)$. Then the submodule $\Fc^C \subset \Fc$ is $G$-invariant.
\end{lemma}
\begin{proof}
Let us view $\Oc_S[G]$ as a $G$-equivariant quasi-coherent $\Oc_S$-module via the coaction morphism
\[
c_G \colon \Oc_S[G] \to \Oc_S[G] \otimes_{\Oc_S} \Oc_S[G] \; \text{corresponding to} \; G \times_S G \to G, (g,h) \mapsto ghg^{-1}.
\]
Since $C \subset G$ is normal, this structure descends to a $G$-equivariant structure on $\Oc_S[C]$, which makes the surjection $\Oc_S[G] \to \Oc_S[C]$ a $G$-equivariant morphism.

To prove the lemma, it will suffice to show that the $C$-coaction morphism $\Fc \to \Oc_S[C] \otimes_{\Oc_S} \Fc$ is $G$-equivariant. Since $\Oc_S[G] \to \Oc_S[C]$ is $G$-equivariant, it will suffice to show that the $G$-coaction morphism $\mu_\Fc \colon \Fc \to \Oc_S[G] \otimes_{\Oc_S} \Fc$ is $G$-equivariant.

Set $\Ac = \Oc_S[G]$, and consider the morphism of $\Oc_S$-modules
\[
m \colon \Ac \otimes_{\Oc_S} \Ac \otimes_{\Oc_S} \Ac \to \Ac \otimes_{\Oc_S} \Ac, \quad a\otimes b \otimes c \mapsto ac \otimes b.
\]
Let us first observe that the composite morphism of $S$-schemes
\[
G \times_S G \xrightarrow{(g,h) \mapsto (g,h,g)} G \times_S G \times_S G \xrightarrow{(g,h,i) \mapsto (ghg^{-1},i)} G \times_S G \xrightarrow{(g,h) \mapsto gh} G
\]
is given by $(g,h) \mapsto gh$. In other words we have an equality of morphisms of $\Oc_S$-modules
\begin{equation}
\label{eq:mu_c_m=mu}
\mu_G = m \circ (c_G \otimes \id_{\Ac}) \circ \mu_G \colon \Ac \to \Ac \otimes_{\Oc_S} \Ac.
\end{equation}
We can now compute, as morphisms $\Fc \to \Ac \otimes_{\Oc_S} \Ac \otimes_{\Oc_S} \Fc$,
\begin{align*}
(\mu_{\Ac \otimes_{\Oc_S} \Fc}) \circ \mu_{\Fc}
&= (m \otimes \id_{\Fc}) \circ (c_G \otimes \id_{\Ac \otimes_{\Oc_S} \Fc}) \circ (\id_{\Ac} \otimes \mu_\Fc) \circ \mu_{\Fc}&&\text{by \rref{p:tensor_equiv}}\\ 
&= (m \otimes \id_{\Fc}) \circ (c_G \otimes \id_{\Ac \otimes_{\Oc_S} \Fc}) \circ (\mu_G \otimes \id_\Fc) \circ \mu_{\Fc} && \text{by \dref{p:def_equiv_module}{p:def_equiv_module:2}}\\
&= \big(\big(m \circ (c_G \otimes \id_{\Ac}) \circ \mu_G\big) \otimes \id_{\Fc}\big) \circ \mu_{\Fc}\\
&= (\mu_G \otimes \id_{\Fc}) \circ \mu_{\Fc} && \text{by \eqref{eq:mu_c_m=mu}}\\
&= (\id_{\Ac} \otimes \mu_{\Fc}) \circ \mu_{\Fc}&& \text{by \dref{p:def_equiv_module}{p:def_equiv_module:2}},
\end{align*}
which shows that $\mu_{\Fc} \colon \Fc \to \Ac \otimes_{\Oc_S} \Fc$ is $G$-equivariant.
\end{proof}

\begin{para}
When $T \to S$ is a scheme morphism with $T$ noetherian, we denote by $G_T=G \times_S T$ the $T$-group scheme obtained by base change. We will often write $\Qcoh^G(T)$ instead of $\Qcoh^{G_T}(T)$, and  $\Fc^G$ instead of $\Fc^{G_T}$ when $\Fc \in \Qcoh^{G_T}(T)$ (see also \rref{p:alt_def_equiv_modules}).
\end{para}

\begin{example}
\label{ex:OS_G}
The group operation in $G$ yields a structure of $G$-equivariant quasi-coherent $\Oc_S$-module on $\Oc_S[G]$, where the coaction morphism is given by the morphism $\mu_G$ of \eqref{eq:Os_G_group}. We claim that
\begin{equation}
\label{eq:fixed_Oc_S_G}
(\Oc_S[G])^G=\Oc_S.
\end{equation}
Indeed set $\Ac=\Oc_S[G]$. Since $\mu_G \circ \iota = (1 \otimes \id_{\Ac}) \circ \iota$, we have $\Oc_S \subset \Ac^G$. Conversely the commutative diagram
\[
\xymatrix{
\Ac \ar[rrr]_\varepsilon \ar[d]_{1 \otimes \id_{\Ac}} &&& \Oc_S\ar[d]^\iota\\
\Ac \otimes_{\Oc_S} \Ac \ar[rr]^-{\id_\Ac \otimes \varepsilon} && \Ac \otimes_{\Oc_S} \Oc_S \ar@{=}[r]& \Ac\\
\Ac \ar[u]^{\mu_G} \ar@/_1pc/[rrru]_{\id_\Ac}
}
\]
shows that
\[
\Ac^G = \ker(1 \otimes \id_\Ac - \mu_G) \subset \ker((\iota \circ \varepsilon) - \id_\Ac)=\Oc_S.
\]
\end{example}

\begin{para}
When $k$ is a noetherian ring and $S=\Spec k$, we will write $k[G]$ instead of $\Oc_S[G]$, and view it as a Hopf algebra over $k$. The notion of $G$-equivariant quasi-coherent $\Oc_S$-module corresponds to that of $k[G]$-comodule. In particular given a $k[G]$-comodule $M$, its fixed $k$-submodule $M^G \subset M$ is defined as in \rref{def:fixed_subsheaf}.
\end{para}

\begin{para}
\label{p:pf_equiv}
Let $f \colon T \to S$ be a scheme morphism with $T$ noetherian, and let $\Fc \in \Qcoh^G(T)$. It follows from the flat base change theorem \cite[Tag~\href{https://stacks.math.columbia.edu/tag/02KH}{02KH}]{stacks} that the morphism of $\Oc_S$-modules $u \colon \Oc_S[G] \otimes_{\Oc_S} f_*\Fc \to f_*(\Oc_T[G_T] \otimes_{\Oc_T} \Fc)$ is an isomorphism. We may thus define a coaction morphism $\mu_{f_*\Fc}\colon f_*\Fc \to \Oc_S[G] \otimes_{\Oc_S} f_*\Fc$ by the condition $u \circ \mu_{f_*\Fc} = f_*(\mu_{\Fc})$. This makes $f_*\Fc$ a  $G$-equivariant quasi-coherent $\Oc_S$-module, which yields a functor $f_* \colon \Qcoh^G(T) \to \Qcoh^G(S)$.
\end{para}

\begin{para}
\label{p:pf_fixed}
In the situation of \rref{p:pf_equiv}, as $u$ is an isomorphism and $f_*$ is left exact, we have
\[
(f_*\Fc)^G = \ker(\mu_{f_*\Fc}) = \ker(u \circ \mu_{f_*\Fc}) = \ker(f_*(\mu_{\Fc})) = f_* \ker(\mu_{\Fc}) = f_*(\Fc^G).
\]
\end{para}

\subsection{Linearly reductive groups}
We continue to assume that $G$ is a group scheme over $S$.
\begin{definition}
\label{def:lin_red}
We say that $G$ is \emph{linearly reductive} if the morphism $\iota \colon \Oc_S \to \Oc_S[G]$ admits a retraction $\rho \colon \Oc_S[G] \to \Oc_S$ in $\Qcoh^G(S)$.
\end{definition}

\begin{para}
\label{p:lin_red_field}
When $S$ is the spectrum of a field, the group $G$ is linearly reductive if and only if the abelian category of $G$-equivariant coherent $\Oc_S$-modules is semisimple (see \cite[Proposition~3.1]{Margaux-red}).
\end{para}

\begin{para}
\label{p:bc_reductive}
If $T \to S$ is a morphism with $T$ noetherian, and $G$ is a linearly reductive group scheme over $S$, then the group scheme $G_T=G \times_S T$ is linearly reductive over $T$.
\end{para}

\begin{para}
\label{p:gamma}
Let $\Fc \in \Qcoh^G(S)$. We will consider the composite in $\Qcoh(S)$
\[
\gamma_\Fc \colon \Fc \xrightarrow{\mu_{\Fc}} \Oc_S[G] \otimes_{\Oc_S} \Fc \xrightarrow{\chi_G \otimes \id_{\Fc}} \Oc_S[G] \otimes_{\Oc_S} \Fc.
\]
(Recall that $\chi_G$ is the antipode of $G$, and $\mu_{\Fc}$ the coaction morphism of $\Fc$.)
\end{para}

\begin{lemma}
\label{lemm:gamma_inv}
Let $\Fc \in \Qcoh^G(S)$. We let $G$ act on $\Oc_S[G]$ via the group operation of $G$ (as in \rref{ex:OS_G}), and on $\Oc_S[G] \otimes_{\Oc_S} \Fc$ as described in \rref{p:tensor_equiv}. Then
\[
\gamma_{\Fc}(\Fc) \subset (\Oc_S[G] \otimes_{\Oc_S} \Fc)^G.
\]
\end{lemma}
\begin{proof}
In this proof, we will write $\times$ instead of $\times_S$, and $\otimes$ instead of $\otimes_{\Oc_S}$. Set $\Ac = \Oc_S[G]$, and consider the morphism of $\Oc_S$-modules
\[
m \colon \Ac \otimes \Ac \otimes \Ac \to \Ac \otimes \Ac, \quad a\otimes b \otimes c \mapsto ac \otimes b.
\]
Let us first observe that the composite morphisms of $S$-schemes
\[
G \times G \xrightarrow{(g,h) \mapsto (g,h,g)} G \times G \times G \xrightarrow{(g,h,i) \mapsto (gh,i)} G \times G \xrightarrow{(g,h) \mapsto (g^{-1},h)} G \times G \xrightarrow{(g,h) \mapsto gh} G
\]
and
\[
G \times G \xrightarrow{(g,h) \mapsto h} G \xrightarrow{g \mapsto g^{-1}} G
\]
are both given by $(g,h) \mapsto h^{-1}$ and thus coincide. In other words, we have an equality of morphisms of $\Oc_S$-modules $\Ac \to \Ac \otimes \Ac$
\begin{equation}
\label{eq:m_chi}
m \circ (\mu_G \otimes \id_\Ac) \circ (\chi_G \otimes \id_\Ac) \circ \mu_G = (1 \otimes \id_\Ac) \circ \chi_G.
\end{equation}
Using the relation $\mu_\Fc \circ (\id_\Ac \otimes \mu_\Fc) = \mu_\Fc \circ (\mu_G \otimes \id_{\Fc})$ (see \dref{p:def_equiv_module}{p:def_equiv_module:2}), we compute
\begin{align*}
\mu_{\Ac \otimes \Fc} \circ \gamma_\Fc
&=(m \otimes \id_{\Fc}) \circ (\mu_G \otimes \mu_\Fc) \circ (\chi_G \otimes \id_{\Fc}) \circ \mu_{\Fc} \\
&=(m \otimes \id_{\Fc}) \circ (\mu_G \otimes \id_{\Ac \otimes \Fc}) \circ (\chi_G \otimes \id_{\Ac \otimes \Fc})\circ (\id_\Ac \otimes \mu_\Fc) \circ \mu_{\Fc} \\ 
&= (m \otimes \id_{\Fc}) \circ (\mu_G \otimes \id_{\Ac \otimes \Fc}) \circ (\chi_G \otimes \id_{\Ac\otimes\Fc})\circ (\mu_G \otimes \id_{\Fc}) \circ \mu_{\Fc} \\
&= \big(\big(m \circ (\mu_G \otimes \id_\Ac) \circ (\chi_G \otimes \id_\Ac) \circ \mu_G\big) \otimes \id_{\Fc}\big) \circ \mu_\Fc\\
&= \big(\big((1 \otimes \id_\Ac) \circ \chi_G\big) \otimes \id_\Fc\big) \circ \mu_\Fc, \quad \quad \text{by \eqref{eq:m_chi}}\\
&= (1 \otimes \id_{\Ac \otimes \Fc}) \circ \gamma_\Fc.\qedhere
\end{align*}
\end{proof}

\begin{proposition}
\label{prop:split_equalizer}
Let $G$ be a linearly reductive group scheme over $S$. Then for every $\Fc \in \Qcoh^G(S)$ the equalizer diagram in $\Qcoh(S)$
\[
\Fc^G \to \Fc \doublerightarrow{\mu_{\Fc}}{1 \otimes \id_{\Fc}} \Oc_S[G] \otimes_{\Oc_S} \Fc
\]
admits a splitting, which is functorial in $\Fc \in \Qcoh^G(S)$.
\end{proposition}
\begin{proof}
The morphism in $\Qcoh(S)$
\[
u_{\Fc} \colon \Oc_S[G] \otimes_{\Oc_S} \Fc \xrightarrow{\chi_G \otimes \id_{\Fc}} \Oc_S[G] \otimes_{\Oc_S} \Fc \xrightarrow{\rho \otimes \id_{\Fc}} \Oc_S \otimes_{\Oc_S} \Fc = \Fc
\]
is a retraction of $1\otimes \id_{\Fc}$. On the other hand, since
\[
u_{\Fc} \circ \mu_{\Fc} = (\rho \otimes \id_{\Fc}) \circ \gamma_{\Fc}
\]
and $\rho \otimes \id_{\Fc}$ is $G$-equivariant, it follows from \rref{lemm:gamma_inv} that $u_{\Fc} \circ \mu_{\Fc}$ factors through the inclusion $\Fc^G \subset \Fc$. This yields the required splitting. The functoriality is clear from that of $u_{\Fc}$ and $\mu_{\Fc}$.
\end{proof}

\begin{para}
\label{p:def_rho}
When $G$ is linearly reductive and $\Fc \in \Qcoh^G(S)$, we will denote by $\rho_{\Fc} \colon \Fc \to \Fc^G$ the retraction in $\Qcoh(S)$ of the inclusion $\Fc^G \subset \Fc$ given by \rref{prop:split_equalizer}. The functoriality \rref{prop:split_equalizer} implies that, for a morphism $f\colon \Fc \to \Gc$ in $\Qcoh^G(S)$ we have in $\Qcoh(S)$
\begin{equation}
\label{eq:funct_rho}
f^G \circ \rho_{\Fc} = \rho_{\Gc} \circ f.
\end{equation}
\end{para}

\begin{corollary}
\label{lemm:pullback_fixed}
Let $G$ be a linearly reductive group scheme over $S$. Let $f \colon T \to S$ be a morphism with $T$ noetherian, and $\Fc \in \Qcoh^G(S)$. Then $f^*(\Fc^G) = (f^*\Fc)^G$.
\end{corollary}

\begin{corollary}
\label{cor:lin_red_inv_exact}
Let $G$ be a linearly reductive group scheme over $S$. Then the functor $\Qcoh^G(S) \to \Qcoh(S)$ given by $\Fc \mapsto \Fc^G$ is exact.
\end{corollary}
\begin{proof}
Let $f \colon \Fc \to \Gc$ be a surjective morphism in $\Qcoh^G(S)$. Then by \eqref{eq:funct_rho} we have $f^G \circ \rho_{\Fc} = \rho_{\Gc} \circ f$. Since $f$ is surjective and $\rho_{\Gc}$ is surjective (it admits a section), it follows that $f^G$ must be surjective.
\end{proof}

\begin{corollary}
\label{cor:line_red_fixed_0}
Let $G$ be a linearly reductive group scheme over $S$. Then for any $\Fc \in \Qcoh^G(S)$, we have $(\Fc/\Fc^G)^G=0$.
\end{corollary}
\begin{proof}
The morphism $\Fc^G \to (\Fc/\Fc^G)^G$ is zero, and it is surjective by \rref{cor:lin_red_inv_exact}.
\end{proof}

\begin{proposition}
\label{prop:unique_retraction}
Let $G$ be a linearly reductive group scheme over $S$. Then the morphism $\iota \colon \Oc_S \to \Oc_S[G]$ admits a unique retraction in $\Qcoh^G(S)$.
\end{proposition}
\begin{proof}
Let $r \colon \Oc_S[G] \to \Oc_S$ be a retraction in $\Qcoh^G(S)$ of $\iota$. Then applying \eqref{eq:funct_rho} to the morphism $r$, and using \rref{eq:fixed_Oc_S_G}, yields a commutative diagram

\[ \xymatrix{
\Oc_S[G]\ar[rr]^-{\rho_{\Oc_S[G]}} \ar[d]_r && \Oc_S \ar[d]^{r^G} \\ 
\Oc_S \ar[rr]^-{\rho_{\Oc_S}} && \Oc_S
}\]
Since $r^G = \rho_{\Oc_S}=\id_{\Oc_S}$, we deduce that $r=\rho_{\Oc_S[G]}$.
\end{proof}

\begin{corollary}
\label{cor:lin_red_is_local}
Let $T \to S$ be a faithfully flat morphism with $T$ noetherian. If the $T$-group scheme $G_T = G \times_S T$ is linearly reductive, then so is the $S$-group scheme $G$.
\end{corollary}
\begin{proof}
This follows from the uniqueness in \rref{prop:unique_retraction} by fpqc descent.
\end{proof}

\begin{remark}
\label{rem:lin_red_finite}
Assume that $G$ is a finite (and flat as usual) group scheme over $S$, and that the functor $\Qcoh^G(S) \to \Qcoh(S)$ given by $\Fc \mapsto \Fc^G$ is exact. We claim that $G$ is linearly reductive.

By \rref{cor:lin_red_is_local} and \cite[Proposition~2.4 (a)]{AOV-Tame}, we may assume that $S$ is affine. Recall that when $\Fc \in \Qcoh^G(S)$ is locally free of finite rank as an $\Oc_S$-module, then for any $\Gc \in \Qcoh^G(S)$ the quasi-coherent $\Oc_S$-module $\Hom_{\Oc_S}(\Fc,\Gc)$ has a natural $G$-equivariant structure, and
\[
\Hom_{\Qcoh^G(S)}(\Fc,\Gc) = H^0(S,\Hom_{\Oc_S}(\Fc,\Gc)^G).
\]
Note that the morphism $\iota \colon \Oc_S \to \Oc_S[G]$ admits a retraction in $\Qcoh(S)$ (denoted $\varepsilon$ above). Thus the morphism in $\Qcoh^G(S)$
\[
\Hom_{\Oc_S}(\Oc_S[G],\Oc_S) \to \Hom_{\Oc_S}(\Oc_S,\Oc_S)
\]
is surjective (since it admits a section in $\Qcoh(S)$). Taking $G$-invariants and then global sections (which are both exact functors) thus shows that
\[
\Hom_{\Qcoh^G(S)}(\Oc_S[G],\Oc_S) \to \Hom_{\Qcoh^G(S)}(\Oc_S,\Oc_S)
\]
is surjective. Any preimage of $\id_{\Oc_S}$ is then a retraction of $\Oc_S \subset \Oc_S[G]$ in $\Qcoh^G(S)$.
\end{remark}

\begin{remark}
\label{rem:lin_red}
Corollary \rref{cor:lin_red_inv_exact} asserts that linearly reductive groups in the sense of \rref{def:lin_red} are linearly reductive in the sense of \cite[Definition~2.14]{Hoyois-equivariant}. The converse holds when $S$ is the spectrum of a field by \rref{p:lin_red_field}, or when $G$ is finite by \rref{rem:lin_red_finite}. In other words, our definition of linear reductivity agrees with that of \cite{Margaux-red} (over fields), and with that of \cite{AOV-Tame} (for finite group schemes).
\end{remark}

\subsection{The zero-locus of a section}

\begin{definition}
\label{def:Z_A}
Let $A$ be a commutative ring and $M$ an $A$-module. Let $\Sigma \subset M$ be a subset. We denote by
\[
\Zei_A(\Sigma) \subset A
\]
the ideal of $A$ generated by the elements $\varphi(t)$, for $t \in \Sigma$ and $\varphi \in \Hom_A(M,A)$.
\end{definition}

\begin{para}
\label{lemm:Z_funct}
In the situation of \rref{def:Z_A}, if $g \colon M \to N$ is a morphism of $A$-modules, then $\Zei_A(g(\Sigma)) \subset \Zei_A(\Sigma)$. In particular if $g$ admits a retraction, then $\Zei_A(g(\Sigma)) = \Zei_A(\Sigma)$
\end{para}

\begin{lemma}
\label{lemm:Z_proj}
Let $A$ be a commutative ring and $M$ a projective $A$-module. Let $\Sigma \subset M$ be a subset.
\begin{enumerate}[(i)]
\item 
\label{lemm:Z_proj:1}
We have $\Zei_A(\Sigma)=0$ if and only if $\Sigma \subset \{0\}$.

\item 
\label{lemm:Z_proj:2}
Let $f\colon A \to B$ be a morphism of commutative rings. Then the ideal $\Zei_B(\Sigma \otimes 1)$ of $B$ is generated by $f(\Zei_A(\Sigma))$. (Here $\Sigma \otimes 1 \subset M \otimes_A B$ denotes the image of $\Sigma$ under the morphism $M \to M \otimes_A B$ induced by $f$.)
\end{enumerate}
\end{lemma}
\begin{proof}
\eqref{lemm:Z_proj:1}: By \rref{lemm:Z_funct}, we may assume that $M$ is free, in which case the statement is clear.

\eqref{lemm:Z_proj:2}: It will suffice to treat the case $\Sigma = \{s\}$, where $s \in M$. By \rref{lemm:Z_funct}, we may assume that $M$ is free. Then $s$ belongs to some finitely generated free $A$-submodule of $M$, and by \rref{lemm:Z_funct} again, we may assume that $M$ is free of finite rank. Then the morphism $\Hom_A(M,A) \otimes_A B \to \Hom_B(M\otimes_A B,B)$ is bijective, and the statement follows easily.
\end{proof}

\begin{definition}
\label{def:zero_scheme}
Let $X$ be a noetherian scheme and $\Fc$ a quasi-coherent $\Oc_X$-module. For $s \in H^0(X,\Fc)$, we let $\Ze_X(s) \subset X$ be the closed subscheme whose ideal is the image of $s^\vee \colon \Fc^\vee \to \Oc_X$. We say that $s$ is a \emph{nowhere vanishing section} of $\Fc$ if $\Ze_X(s) = \varnothing$. 
\end{definition}

\begin{para}
\label{p:Ze:funct}
When $\varphi \colon \Fc \to \Gc$ is a morphism of quasi-coherent $\Oc_X$-modules, we have
\[
\Ze_X(s) \subset \Ze_X(\varphi(s)).
\]
\end{para}

\begin{para}
\label{p:Z_affine}
When $X=\Spec A$ is affine (and noetherian) in \rref{def:zero_scheme}, the ideal of $\Ze_X(s)$ is $\Zei_A(\{s\})$, in the notation of \rref{def:Z_A}.
\end{para}

\begin{para}
Let $X$ be a noetherian scheme and $V \to X$ a vector bundle. The sheaf of sections $\Vc$ of $V$ is a locally free finite rank $\Oc_X$-module. When $s \in H^0(X,\Vc)$, we have a cartesian square of schemes, where $z$ denotes the zero-section,
\[ \xymatrix{
\Ze_X(s)\ar[r] \ar[d] & X \ar[d]^z \\ 
X \ar[r]^s & V
}\]
In particular the section $s \in H^0(X,\Vc)$ is nowhere vanishing if and only if the corresponding morphism $s\colon X \to V$ factors through the complement of the zero-section $V^\circ=V \smallsetminus z(X)$.
\end{para}

\begin{lemma}
\label{lemm:Ze_funct}
Let $f \colon Y \to X$ be a morphism of noetherian schemes, and $\Fc$ be a quasi-coherent $\Oc_X$-module. Then for any $s \in H^0(X,\Fc)$, we have
\[
\Ze_Y(f^*(s)) \subset f^{-1}\Ze_X(s).
\]
\end{lemma}
\begin{proof}
This follows from the fact that $f^*(\Fc^\vee) \xrightarrow{f^*(s^\vee)} f^*\Oc_X \simeq \Oc_Y$ factors as $f^*(\Fc^\vee) \to (f^*\Fc)^\vee \xrightarrow{(f^*s)^\vee} \Oc_Y$.
\end{proof}

\begin{para}
\label{p:alt_def_equiv_modules}
Let $G$ be a group scheme over $S$. When $X$ is an $S$-scheme with a $G$-action, we will consider the category $\Qcoh(X;G)$ of quasi-coherent $G$-equivariant $\Oc_X$-modules, for which we refer to \cite[I, \S6.5]{SGA3-1}. We note that the category $\Qcoh(S;G)$ is naturally equivalent to the category $\Qcoh^G(S)$ defined in \rref{p:def_equiv_module} (see \cite[I, Proposition~4.7.2 and Remarque~6.5.4]{SGA3-1}). Moreover if $T \to S$ is a scheme morphism with $T$ noetherian (carrying the trivial $G$-action), then we have natural equivalences $\Qcoh(T;G) \simeq \Qcoh(T;G_T) \simeq \Qcoh^{G}(T)$.
\end{para}

\begin{para}
\label{p:def_G_eq_section}
Let $G$ be a group scheme over $S$. Let $X$ be a noetherian $S$-scheme with a $G$-action and $\Fc \in \Qcoh(X;G)$. A \emph{$G$-equivariant section} of $\Fc$ is an element of
\[
\Hom_{\Qcoh(X;G)}(\Oc_X,\Fc) \subset \Hom_{\Qcoh(X)}(\Oc_X,\Fc).
\]
Writing $\pi=\pi_X \colon X \to S$, the subset of $G$-equivariant sections of $\Fc$ may be identified with 
\[
H^0(S,(\pi_*\Fc)^G) \subset H^0(S,\pi_*\Fc) = H^0(X,\Fc).
\]
\end{para}

\begin{para}
\label{p:equiv_vb}
Let $X$ be a noetherian $S$-scheme with a $G$-action. Let $E \to X$ be a vector bundle, and $\Ec$ its sheaf of sections. We say that $E$ is a \emph{$G$-equivariant vector bundle} if $\Ec$ is a quasi-coherent $G$-equivariant $\Oc_X$-module. Such a structure induces a $G$-action on $E$ such that the morphism $E\to X$ is $G$-equivariant (and in addition the action is ``linear'' on the fibres, see \cite[I, Remarque~6.5.3]{SGA3-1}).
\end{para}

\begin{para}
\label{p:eq_sections}
In the situation of \rref{p:equiv_vb}, a section of $E$ is $G$-equivariant as a scheme morphism $X \to E$ if and only if the corresponding section of $\Ec$ is $G$-equivariant in the sense of \rref{p:def_G_eq_section}.
\end{para}

\begin{lemma}
\label{lemm:fixed_contained}
Let $C$ be a linearly reductive group scheme over $S$ (see \rref{def:lin_red}). Let $\Fc$ be a $C$-equivariant quasi-coherent $\Oc_S$-module such that $\Fc^C=0$. Let $X$ be a noetherian $S$-scheme with a $C$-action, and $Y \subset X$ a $C$-equivariant closed subscheme on which $C$ acts trivially. If $s$ is a $C$-equivariant section of $\pi_X^*\Fc$, then $Y \subset \Ze_X(s)$ as closed subschemes of $X$.
\end{lemma}
\begin{proof}
Denote by $i \colon Y \to X$ the closed immersion, and write $\pi=\pi_Y \colon Y \to S$. The pullback $i^*(s)$ is a $C$-equivariant section of $i^*\pi_X^*\Fc = \pi^*\Fc$ over $Y$. Now we have
\[
(\pi_*i^*\pi_X^*\Fc)^C = (\pi_*\pi^*\Fc)^C \overset{\rref{p:pf_fixed}}{=} \pi_*((\pi^*\Fc)^C) \overset{\rref{lemm:pullback_fixed}}{=} \pi_*\pi^*(\Fc^C) = 0.
\]
We deduce that $i^*(s) = 0$. Thus, by \rref{lemm:Ze_funct}, it follows that
\[
Y = \Ze_Y(i^*(s)) \subset i^{-1}\Ze_X(s) = Y \cap \Ze_X(s) \subset \Ze_X(s).\qedhere
\]
\end{proof}

\subsection{The fixed locus}
\begin{para}
\label{p:def_loc_proj}
(See \cite[Tag~\href{https://stacks.math.columbia.edu/tag/05JN}{05JN}]{stacks}.) A quasi-coherent $\Oc_S$-module $\Fc$ is called \emph{locally projective} if $\Fc(U)$ is a projective $\Oc_S(U)$-module for every affine open subscheme $U \subset S$.
\end{para}

\begin{para}
\label{p:bc_loc_proj}
(See \cite[Tag~\href{https://stacks.math.columbia.edu/tag/060M}{060M}]{stacks}.) If the quasi-coherent $\Oc_S$-module $\Fc$ is locally projective, and $f \colon T \to S$ is a morphism with $T$ noetherian, then the $\Oc_T$-module $f^*\Fc$ is locally projective.
\end{para}

\begin{para}
\label{p:def_fixed_locus}
Let $C$ be a group scheme over $S$ such that the $\Oc_S$-module $\Oc_S[C]$ is locally projective (see \rref{p:def_loc_proj}). If $X$ is a separated $S$-scheme with a $C$-action, then there exists a closed subscheme $X^C \subset X$ such that for any $S$-scheme $T$, the subset $X^C(T)\subset X(T)$ consists of those morphisms $T\to X$ which are $C$-equivariant with respect to the trivial $C$-action on $T$. This is proved in \cite[Proposition~A.8.10~(1)]{CGP} under the assumption that $S$ is affine, and the general case follows easily. (A proof when $X$ and $S$ are affine is given in \rref{lemm:fixed_locus} below.)
\end{para}

\begin{remark}
We have seen in \rref{p:bc_loc_proj} that the assumption that the $\Oc_S$-module $\Oc_S[C]$ is locally projective is stable under base change. It is satisfied under any of the following assumptions:
\begin{itemize}
\item $S$ is the spectrum of a field,
\item $C$ is of multiplicative type,
\item $C$ is finite constant,
\item the geometric fibres of $C$ over $S$ are integral (see \cite[Chapter 2, \S1, Theorem~1 and Example~3]{Raynaud-flat}).
\end{itemize}
We refer to \cite[Appendix]{Margaux-red} for a discussion of a closely related condition.
\end{remark}

\begin{para}
\label{p:fixed_bc_mono}
In the situation of \rref{p:def_fixed_locus}, let $f\colon Y \to X$ be a $C$-equivariant morphism between separated $S$-schemes with a $C$-action. If $f$ is a scheme monomorphism (for instance a closed immersion), then $Y^C=X^C \times_X Y$.
\end{para}

\begin{para}
Let $C$ be a linearly reductive group scheme over $S$, and assume that the $\Oc_S$-module $\Oc_S[C]$ is locally projective. Let $E \to S$ be a $C$-equivariant vector bundle, with sheaf of sections $\Ec$. Then the $\Oc_S$-module $\Ec$ is locally free of finite rank, and it follows from \rref{prop:split_equalizer} that the same is true for $\Ec^C$. For any scheme morphism $f \colon T \to S$ with $T$ noetherian (and trivial $C$-action), we have
\begin{align*}
E^C(T)  
&= \Hom_{\Qcoh^C(T)}(\Oc_T,f^*\Ec)  && \text{by \rref{p:eq_sections}}\\ 
&=  \Hom_{\Qcoh(T)}(\Oc_T,(f^*\Ec)^C) && \text{by  \rref{p:fp_adjoint}}\\
&= H^0(T,(f^*\Ec)^C)\\
&=H^0(T,f^*(\Ec^C)) && \text{by \rref{lemm:pullback_fixed}}.
\end{align*}
It follows that $E^C \to S$ is naturally a $C$-equivariant vector bundle, and that its sheaf of sections is $\Ec^C$.
\end{para}

\begin{lemma}
\label{p:smooth_fixed_locus}
Let $C$ be a linearly reductive group scheme over $S$, and assume that the  $\Oc_S$-module $\Oc_S[C]$ is locally projective. Let $X$ be a smooth separated $S$-scheme with a $C$-action. Then $X^C$ is smooth over $S$. Denoting by $N$ the normal bundle to the closed immersion $i \colon X^C \to X$, we have $N^C=0$ as a vector bundle over $X^C$.
\end{lemma}
\begin{proof}
We may assume that $S$ is affine. Then the first statement follows from \cite[Proposition~A.8.10~(2)]{CGP}, taking \rref{p:lin_red_field} and \rref{p:bc_reductive} into account. For any $S$-schemes $V,W$ equipped with $C$-actions, let us denote by $\Mor_{C-\text{eq}}(V,W)$ the set of $C$-equivariant morphisms $V\to W$. Then, for any $S$-scheme $U$ (with trivial $C$-action), writing $\tilde{U} = U \times_{\Spec (\Zz)} \Spec(\Zz[\varepsilon]/\varepsilon^2)$, we have
\[
(\Tan_{X/S})^C(U) = \Mor_{C-\text{eq}}(U,\Tan_{X/S}) = \Mor_{C-\text{eq}}(\tilde{U},X) = X^C(\tilde{U}) = \Tan_{X^C/S}(U),
\]
showing that $\Tan_{X^C/S} = (\Tan_{X/S})^C$ (see also \cite[Proposition~A.8.10~(1)]{CGP}). Now, consider the exact sequence of $C$-equivariant vector bundles over $X^C$ (see e.g.\ \cite[Tag~\href{https://stacks.math.columbia.edu/tag/06AA}{06AA}]{stacks})
\[
0 \to \Tan_{X^C/S} \to i^*(\Tan_{X/S}) \to N \to 0.
\]
As $(i^*(\Tan_{X/S}))^C = (\Tan_{X/S})^C$ by \rref{p:fixed_bc_mono}, it follows from \rref{cor:line_red_fixed_0} that $N^C=0$.
\end{proof}

\section{Geometric concentration}
\numberwithin{theorem}{section}
\numberwithin{lemma}{section}
\numberwithin{proposition}{section}
\numberwithin{corollary}{section}
\numberwithin{example}{section}
\numberwithin{notation}{section}
\numberwithin{definition}{section}
\numberwithin{remark}{section}

We will use the following description of the ideal of the fixed locus in the affine case:
\begin{proposition}
\label{lemm:fixed_locus}
Let $k$ be a commutative noetherian ring, and $C$ a group scheme over $k$ such that the $k$-module $k[C]$ is projective. Let $A$ be a commutative $k$-algebra, and assume given a $C$-action on $X = \Spec A$. Denote by $\alpha_C \colon A \to k[C] \otimes_k A$ the corresponding morphism.

Then the subfunctor $X^C \subset X$ is represented by the closed subscheme defined by the ideal $\Zei_A(\rho_C(A))$ of $A$ (see \rref{def:Z_A}), where
\[
\rho_C = \alpha_C - 1 \otimes \id_A \colon A \to k[C] \otimes_k A.
\]
\end{proposition}
\begin{proof}
Let $S= \Spec k$. Let $x \colon T \to X$ be a morphism of $S$-schemes. Then $x \in X(T)$ belongs to $X^C(T)$ if and only if the composites 
\[
C \times_S T \xrightarrow{\id_C \times x} C \times_S X  \doublerightarrow{a}{p} X
\]
coincide, where $a,p$ are the action and projection morphisms respectively.

Now, the coequalizer of the morphisms of commutative $k$-algebras $\alpha_C, 1 \otimes \id_A \colon A \to k[C] \otimes_k A$ is $(k[C] \otimes_k A)/J$, where $J  \subset k[C] \otimes_k A$ is the ideal generated by the image of the morphism $\rho_C = \alpha_C - 1 \otimes \id_A$. Therefore the equalizer of the morphisms $a,p$ in the category of affine $S$-schemes is the closed subscheme of $C \times_S X$ given by the ideal $J$.

Thus a morphism of commutative $k$-algebras $f\colon A \to B$ belongs to $X^C(\Spec B) \subset X(\Spec B)$ if and only if the morphism $\id_{k[C]} \otimes f \colon k[C] \otimes_k A \to k[C] \otimes_k B$ maps $J$ to $0$. But by \rref{lemm:Z_proj} we have
\[
(\id_{k[C]} \otimes f)(J) = 0 \Longleftrightarrow \Zei_B((\id_{k[C]} \otimes f)(J)) =0 \Longleftrightarrow f(\Zei_A(J))=0.
\]
We deduce that the closed subscheme of $X$ given by the ideal $\Zei_A(J) \subset A$ represents the subfunctor $X^C \subset X$ in the category of affine $S$-schemes, and it follows easily that it does so in the category of $S$-schemes.
\end{proof}

The next statement may be viewed as a form of converse of \rref{lemm:fixed_contained}, and is the technical heart of this paper.

\begin{proposition}
\label{prop:main_affine}
Let $k$ be a commutative noetherian ring and $S=\Spec k$. Let $G$ be a group scheme over $S$, and $C \subset G$ a closed normal subgroup scheme. Assume that the $k$-module $k[C]$ is projective, and that the group scheme $C$ is linearly reductive \rref{def:lin_red}. Let $X=\Spec A$ be an affine noetherian $S$-scheme, equipped with a $G$-action.

Then there exists a $k[G]$-comodule $W$ such that $W^C=0$, and an element $s \in (W \otimes_k A)^G \subset W \otimes_k A$ such that $X^C=\Ze_X(s)$, as closed subschemes of $X$.

If $G$ is linearly reductive, or if $X$ is flat over $S$, then $W$ may be chosen to be of finite type as a $k$-module.
\end{proposition}
\begin{proof}
We view $k[G]$, resp.\ $k[C]$, as a $k[G]$-comodule, resp.\ $k[C]$-comodule, via the coaction given by the group operation (as in \rref{ex:OS_G}). Recall from \eqref{eq:fixed_Oc_S_G} that $(k[C])^C=k$. Consider the exact sequences of $k[G]$-comodules and $k[C]$-comodules respectively (see \rref{lemm:inv_is_equ}), which define $P_G$ and $P_C$,
\[
0 \to (k[G])^C \to k[G] \xrightarrow{\pi_G} P_G \to 0 \quad\text{and} \quad 0 \to k \to k[C] \xrightarrow{\pi_C} P_C \to 0.
\]
As $C$ is linearly reductive, it follows from \rref{cor:line_red_fixed_0} that
\begin{equation}
\label{eq:P_G_fixed}
(P_G)^C =0.
\end{equation}
The inclusion $C \subset G$ induces morphisms of $k$-modules
\[
\varphi \colon k[G] \to k[C] \quad \text{and} \quad \psi \colon P_G \to P_C.
\]
The $G$-action, resp.\ $C$-action, on $X$ is given by a morphism of $k$-algebras $\alpha_G \colon A \to k[G] \otimes_k A$, resp.\ $\alpha_C \colon A \to k[C] \otimes_k A$, making $A$ a $k[G]$-comodule, resp.\ $k[C]$-comodule. Consider the composite
\begin{equation}
\label{eq:def_sigma_G}
\sigma_G\colon A \xrightarrow{\alpha_G} k[G] \otimes_k A \xrightarrow{\chi_G \otimes \id_A} k[G] \otimes_k A\xrightarrow{\pi_G \otimes \id_A} P_G \otimes_k A,
\end{equation}
where $\chi_G \colon k[G] \to k[G]$ is the antipode of the Hopf algebra $k[G]$. Note that the composite $(\chi_G \otimes \id_A) \circ \alpha_G$ is the morphism denoted $\gamma_A$ in \rref{p:gamma}. Since $\pi_G \otimes \id_A$ is a morphism of $k[G]$-comodules, it follows from \rref{lemm:gamma_inv} that
\begin{equation}
\label{eq:sigma_equiv}
\sigma_G(A) \subset (P_G \otimes_k A)^G.
\end{equation}

Let $J \subset A$ be the ideal defining the closed subscheme $X^C \subset X$. It follows from \rref{lemm:fixed_locus} that (in the notation of \rref{def:Z_A})
\begin{equation}
\label{eq:Z_rho_C}
J = \Zei_A(\rho_C(A)), \quad \text{where } \rho_C= \alpha_C -1 \otimes \id_A \colon A \to k[C] \otimes_k A.
\end{equation}
The counit $\varepsilon_C \colon k[C] \to k$ is a $k$-linear retraction of $\iota_C \colon k \to k[C]$, hence it induces a $k$-linear section $\zeta \colon P_C \to k[C]$ of $\pi_C$, which verifies $\zeta \circ \pi_C +  \iota_C \circ\varepsilon_C= \id_{k[C]}$. Now it follows from \dref{p:def_equiv_module}{p:def_equiv_module:1} that the composite
\[
A \xrightarrow{\alpha_C} k[C] \otimes_k A \xrightarrow{\varepsilon_C \otimes \id_A} k \otimes_k A \xrightarrow{\iota_C \otimes \id_A} k[C] \otimes_k A
\]
equals $1 \otimes \id_A$, so that
\[
\rho_C=\alpha_C - 1\otimes \id_A = ((\id_{k[C]}-\iota_C \circ\varepsilon_C)\otimes \id_A) \circ \alpha_C =  ((\zeta \circ \pi_C) \otimes \id_A) \circ \alpha_C.
\]
As $\alpha_C = (\varphi \otimes \id_A) \circ \alpha_G$, we obtain
\begin{equation}
\label{eq:rho_alpha}
\rho_C = ((\zeta \circ \pi_C \circ \varphi) \otimes \id_A) \circ \alpha_G.
\end{equation}

Next, the commutative diagrams of $S$-schemes (recall that $C \subset G$ is normal)
\[ \xymatrix{
C \times_S G \ar[rr]^-{(c,g) \mapsto g}  \ar[d]_{(c,g) \mapsto (g^{-1}c^{-1}g,g^{-1})} && G \ar[d]^{g \mapsto g^{-1}} &&&& C \times_S G \ar[rr]^-{(c,g) \mapsto cg}  \ar[d]_{(c,g) \mapsto (g^{-1}c^{-1}g,g^{-1})} && G \ar[d]^{g \mapsto g^{-1}}\\ 
C \times_S G \ar[rr]^-{(c,g) \mapsto g}  && G &&&& C \times_S G  \ar[rr]^-{(c,g) \mapsto cg} && G
}\]
yield a commutative diagram of $k$-modules, where $\beta$ is the $C$-coaction morphism of $k[G]$, and $\tau$ is induced by the morphism $C \times_S G \to C \times_S G, (c,g) \mapsto (g^{-1}c^{-1}g,g^{-1})$,
\[ \xymatrix{
k[G]\ar[rrr]^-{\beta - 1 \otimes \id_{k[G]}} \ar[d]_{\chi_G} &&& k[C] \otimes_k k[G] \ar[d]^\tau \\ 
k[G] \ar[rrr]^-{\beta - 1 \otimes \id_{k[G]}} &&& k[C] \otimes_k k[G]
}\]
It follows that the morphism $\chi_G$ maps $k[G]^C$ into $k[G]^C$, hence descends to a morphism of $k$-modules $\widetilde{\chi}_G \colon P_G \to P_G$ such that $\pi_G \circ \chi_G = \widetilde{\chi}_G \circ \pi_G$. Thus
\begin{equation}
\label{eq:double_antipode}
\pi_C \circ \varphi = \psi \circ \pi_G = \psi \circ \pi_G \circ \chi_G \circ \chi_G = \psi \circ \widetilde{\chi_G} \circ \pi_G \circ \chi_G.
\end{equation}

Combining \eqref{eq:rho_alpha}, \eqref{eq:double_antipode} and \eqref{eq:def_sigma_G}, we deduce that
\begin{equation}
\label{eq:rho_sigma}
\rho_C = ((\zeta \circ \psi \circ \widetilde{\chi_G}) \otimes \id_A) \circ \sigma_G.
\end{equation}
Thus, in view of \rref{lemm:Z_funct}, we have $\Zei_A(\rho_C(A)) \subset \Zei_A(\sigma_G(A))$. Using \eqref{eq:Z_rho_C} and \eqref{eq:sigma_equiv}, we deduce that $J \subset \Zei_A((P_G \otimes_k A)^G)$. On the other hand, by \rref{eq:P_G_fixed} and \rref{lemm:fixed_contained} (in view of \rref{p:Z_affine}) we have $\Zei_A((P_G \otimes_k A)^C) \subset J$. As $\Zei_A((P_G \otimes_k A)^G) \subset \Zei_A((P_G \otimes_k A)^C)$, we conclude that
\[
J = \Zei_A((P_G \otimes_k A)^G).
\]
Since the ideal $J \subset A$ is finitely generated, there exist $s_1,\dots,s_n \in (P_G \otimes_k A)^G$ such that $\Zei_A(\{s_1,\dots,s_n\})=J$. Letting $W = (P_G)^n$ and $s=(s_1,\dots,s_n) \in (W \otimes_k A)^G$ we have $\Zei_A(\{s\})=J$, and by \eqref{eq:P_G_fixed} we have $W^C=0$. This concludes the proof of the first statement.

Finally since $W$ is the union of its finitely generated $k[G]$-comodules \cite[\S1.5]{Serre-Groupes_Grothendieck}, it follows that there exists a finitely generated $k[G]$-subcomodule $W_0 \subset W$ such that $s$ belongs to the image $Q$ of the induced morphism $W_0 \otimes_k A \to W \otimes_k A$. If $G$ is linearly reductive, then the morphism $(W_0 \otimes_k A)^G \to Q^G$ is surjective by \rref{cor:lin_red_inv_exact}; if instead $A$ is flat over $k$, then $W_0\otimes_k A \to Q$ is an isomorphism, hence so is $(W_0 \otimes_k A)^G \to Q^G$. In any case we find an element $s_0 \in (W_0 \otimes_k A)^G$ mapping to $s \in Q^G \subset (W \otimes_k A)^G$. Note that $(W_0)^C=0$, because $W_0 \subset W$. We have $\Zei_A(\{s\}) \subset \Zei_A(\{s_0\})$ by \rref{lemm:Z_funct}, while $\Zei_A(\{s_0\}) \subset J$ by \rref{lemm:fixed_contained}. As $J= \Zei_A(\{s\})$, we conclude that $\Zei_A(\{s_0\}) = J$.
\end{proof}

\begin{definition}[See {\cite[Definition~2.1]{Thomason-resolution}}]
\label{def:resolution}
A group scheme $G$ over $S$ is said to have the \emph{resolution property} if every $G$-equivariant coherent $\Oc_S$-module is a $G$-equivariant quotient of a $G$-equivariant locally free of finite rank $\Oc_S$-module.
\end{definition}

We refer to \cite[\S2]{Thomason-resolution} for a discussion of cases where the resolution property holds (which include the case when $S$ is the spectrum of a field).

\begin{lemma}
\label{lemm:quotien_free}
Let $G$ be a group scheme over $S$, and $C \subset G$ a closed normal subgroup. Assume that $G$ has the resolution property and that $C$ is linearly reductive \rref{def:lin_red}. Let $X$ be a noetherian quasi-affine scheme over $S$, equipped with a $G$-action such that $C$ acts trivially on $X$. Let $\Fc$ be a $G$-equivariant coherent $\Oc_X$-module such that $\Fc^C=0$. Then there exists a $G$-equivariant locally free of finite rank $\Oc_S$-module $\Ec$ such that $\Ec^C=0$ and a surjective morphism of $G$-equivariant quasi-coherent $\Oc_X$-modules $\pi_X^*\Ec \to \Fc$.
\end{lemma}
\begin{proof}
Let us write $\pi=\pi_X \colon X \to S$. The $G$-equivariant quasi-coherent $\Oc_S$-module $\pi_*\Fc$ is the colimit of its $G$-equivariant coherent submodules \cite[Lemma~2.10]{Hoyois-equivariant} (see also \cite[Lemma~1.4]{Thomason-resolution}). Since $G$ has the resolution property, it follows that there exists a family of $G$-equivariant locally free of finite rank $\Oc_S$-modules $\Ec'_\beta$, together with morphisms of $G$-equivariant $\Oc_S$-modules $\Ec'_\beta \to \pi_*\Fc$, such that the induced morphism $\bigoplus_{\beta} \Ec'_\beta \to \pi_*\Fc$ is surjective. As $(\pi_*\Fc)^C=\pi_*(\Fc^C)=0$ by \rref{p:pf_fixed}, this morphism factors through a surjective morphism $\bigoplus_{\beta} \Ec_\beta \to \pi_*\Fc$, where $\Ec_\beta = \Ec'_\beta/(\Ec'_\beta)^C$. Since $C$ is linearly reductive, we have $(\Ec_\beta)^C=0$ by \rref{cor:line_red_fixed_0}. Moreover it follows from \rref{prop:split_equalizer} that $(\Ec'_\beta)^C$ is a direct summand of $\Ec_\beta'$ in $\Qcoh(S)$, and in particular that the $\Oc_S$-module $\Ec_\beta$ is locally free of finite rank. Since $\pi$ is quasi-affine, the natural morphism of $\Oc_X$-modules $\pi^*\pi_* \Fc \to \Fc$ is surjective. We thus have a surjective morphism $\bigoplus_{\beta} \pi^*\Ec_\beta \to \pi^*\pi_*\Fc \to \Fc$. Since $\Fc$ is coherent and $X$ is noetherian, there exists a finite set of indices $B$ such that $\bigoplus_{\beta \in B} \pi^*\Ec_\beta \to \Fc$ is surjective. Then the $\Oc_S$-module $\Ec= \bigoplus_{\beta \in B} \Ec_\beta$ is locally free of finite rank, and we have $\Ec^C=\bigoplus_{\beta \in B} (\Ec_\beta)^C=0$.
\end{proof}

\begin{theorem}
\label{th:conc_main}
Let $f\colon T \to S$ be a quasi-affine morphism, where $T$ is affine and noetherian. Let $G$ be a group scheme over $S$, which is linearly reductive \rref{def:lin_red} and has the resolution property \rref{def:resolution}. Let $C$ be a closed normal subgroup scheme of $G$, which is linearly reductive. Assume that the $\Oc_S$-module $\Oc_S[C]$ is locally projective. Let $X$ be an affine noetherian $T$-scheme with a $G$-action.

Then there exists a $G$-equivariant vector bundle $V\to S$ such that $V^C=0$, and a $G$-equivariant section $s$ of the pullback of $V$ to $X$ such that $X^C = \Ze_X(s)$.
\end{theorem}
\begin{proof}
Let us denote by $\pi \colon X \to T$ the structure morphism. Since $T$ is affine, it follows from \rref{p:bc_loc_proj} that the $\Oc_T$-module $\Oc_T[C_T]$ is projective. By \rref{prop:main_affine} we may find a $G$-equivariant coherent $\Oc_T$-module $\Fc$ such that $\Fc^C=0$, and a $G$-equivariant section $t$ of $\pi^*\Fc$ over $X$ such that $\Ze_X(t)=X^C$. By \rref{lemm:quotien_free} we may find a $G$-equivariant locally free of finite rank $\Oc_S$-module $\Ec$ such that $\Ec^C=0$, together with a surjective morphism of $G$-equivariant quasi-coherent $\Oc_T$-modules $f^*\Ec \to \Fc$. As $\pi$ is affine, the morphism of $\Oc_S$-modules $\pi_*\pi^*f^*\Ec \to \pi_*\pi^*\Fc$ is surjective, hence so is $(\pi_*\pi^*f^*\Ec)^G \to (\pi_*\pi^*\Fc)^G$ by \rref{cor:lin_red_inv_exact}. As $T$ is affine, this implies that the section $t$ lifts to a $G$-equivariant section $s$ of $\pi^*f^* \Ec$ over $X$. It follows from \rref{p:Ze:funct} that $\Ze_X(s) \subset \Ze_X(t)=X^C$. On the other hand, noting that $(f^*\Ec)^C=f^*(\Ec^C)=0$ by \rref{lemm:pullback_fixed}, we have $X^C \subset \Ze_X(s)$ by \rref{lemm:fixed_contained}, and so we conclude that $X^C=\Ze_X(s)$. Thus we can let $V \to S$ be the $G$-equivariant vector bundle whose sheaf of sections is $\Ec$.
\end{proof}

\begin{remark}
\label{rem:unipotent}
Let $G$ be a nontrivial group scheme over $S$. Assume that $G$ is unipotent, i.e.\ that the only $G$-equivariant vector bundle $V \to S$ such that $V^G=0$ is $V=0$. Set $X=G$ with the $G$-action given by left multiplication, so that $X^G=\varnothing$. This shows that \rref{th:conc_main} fails when $G$ is nontrivial and unipotent, instead of linearly reductive.
\end{remark}

\section{Motivic concentration}
\label{sect:motivic}
\numberwithin{theorem}{subsection}
\numberwithin{lemma}{subsection}
\numberwithin{proposition}{subsection}
\numberwithin{corollary}{subsection}
\numberwithin{example}{subsection}
\numberwithin{notation}{subsection}
\numberwithin{definition}{subsection}
\numberwithin{remark}{subsection}

Let $S$ be a noetherian scheme, and $G$ a (flat, affine and finite type) group scheme over $S$. From now on, we will assume that $G$ is linearly reductive \rref{def:lin_red} and has the resolution property \rref{def:resolution}.

\subsection{Euler classes}
\begin{para}
\label{def:Sch}
We denote by $\Sch^G_S$ the full subcategory of the category of $G$-equivariant schemes over $S$ whose objects are the $G$-quasi-projective $S$-schemes \cite[Definition~2.5~(2)]{Hoyois-equivariant}. We denote by $\Sm^G_S \subset \Sch^G_S$ the full subcategory whose objects are additionally smooth over $S$. When $G=1$ is the trivial group, we will write $\Sch_S,\Sm_S$ instead of $\Sch^G_S,\Sm^G_S$.
\end{para}

\begin{para}
When $X \in \Sch^G_S$, we will denote by $\SH^G(X)$ the homotopy category of the stable $\infty$-category of motivic $G$-spectra over $X$ defined by Hoyois in \cite[\S7]{Hoyois-equivariant}. (Note that by \rref{rem:lin_red}, the $S$-group scheme $G$ is tame in the sense of \cite[Definition~2.26]{Hoyois-equivariant}.)
\end{para}

\begin{remark}
The category $\SH^G(X)$ is more generally defined when $X$ is only assumed Nisnevich-locally $G$-quasi-projective over $S$, and most of the results below extend to this setting. In the interest of simplicity, we will only consider $G$-quasi-projective schemes over $S$ in this paper.
\end{remark}

\begin{definition}
\label{def:Euler}
Let $V$ be a $G$-equivariant vector bundle over $X \in \Sch^G_S$, and $z \colon X \to V$ its zero-section. Let $V^\circ = V \smallsetminus z(X)$. Consider the Thom space $\Th V = V/V^\circ$ (as pointed motivic $G$-space \cite[\S3]{Hoyois-equivariant}). We will write
\[
\Sph^V=\Su \Th V\in \SH^G(X).
\]
Setting for $A\in \SH^G(X)$
\[
\Sigma^V A = \Sph^V \wedge A
\]
defines a functor $\Sigma^V \colon \SH^G(X) \to \SH^G(X)$. We consider the morphism in $\SH^G(X)$
\[
e_V \colon \Un_X=\Sup X \xrightarrow{\Su_+ z} \Su_+ V \to (\Su_+ V)/(\Su_+ V^\circ) = \Su \Th V =\Sph^V.
\]
For $A \in \SH^G(X)$, we will write in $SH^G(X)$
\[
e_{V,A} \colon A \simeq \Un_X \wedge A \xrightarrow{e_V \wedge \id_A} \Sph^V \wedge A \simeq \Sigma^VA.
\]
We will also use this notation when $A$ is a functor taking values in $\SH^G(X)$. In an attempt to lighten the notation, when $f \colon Y \to X$ is a $G$-equivariant morphism we will often write $\Sigma^V,e_V,e_{V,A}$ instead of $\Sigma^{f^*V},e_{f^*V},e_{f^*V,A}$.
\end{definition}

\begin{para}
\label{p:triangle_euler}
In the situation of \rref{def:Euler}, let us denote by $\pi \colon V^\circ \to X$ the projection. Then we have a distinguished triangle in $\SH^G(X)$
\[
\Sup V^\circ \xrightarrow{\Sup \pi} \Un_X \xrightarrow{e_V} \Sph^V \to \Sigma^{1,0} \Sup V^{\circ}.
\]
\end{para}

\begin{para}
\label{p:Euler_sum}
Let $V,W$ be $G$-equivariant vector bundles over $X \in \Sch^G_S$. The natural isomorphism of motivic $G$-spaces $\Th V \wedge \Th W \simeq \Th(V \oplus W)$ induces an isomorphism
\begin{equation}
\label{eq:Th_sum}
\Sph^V \wedge \Sph^W \simeq \Sph^{V \oplus W}\in \SH^G(X),
\end{equation}
under which the morphism $e_V \wedge e_W$ corresponds to $e_{V \oplus W}$.

The isomorphism \eqref{eq:Th_sum} induces an isomorphism of functors $\Sigma^V \circ \Sigma^W \simeq \Sigma^{V \oplus W}$, and for any $A \in \SH^G(X)$ the morphism $e_{V \oplus W,A}$ factors as
\[
A \xrightarrow{e_{V,A}} \Sigma^VA \xrightarrow{\Sigma^V(e_{W,A})} \Sigma^V\Sigma^W A \simeq \Sigma^{V \oplus W}A.
\]
\end{para}

\begin{lemma}
\label{lemm:e_section}
Let $V$ be a $G$-equivariant vector bundle over $X \in \Sch^G_S$. If $V$ admits a $G$-equivariant nowhere vanishing section, then $e_V=0$ in $\SH^G(X)$.
\end{lemma}
\begin{proof}
The given section factors as a section of the projection $V^\circ \to X$, hence the lemma follows from the distinguished triangle of \rref{p:triangle_euler}.
\end{proof}

\begin{lemma}
\label{lemm:Euler_closed_open}
Let $X \in \Sch^G_S$. Let $Y$ be a $G$-invariant closed subscheme of $X$, and $U$ its open complement. If $\alpha,\beta$ are morphisms in $\SH^G(X)$ such that $\alpha$ restricts to zero in $\SH^G(Y)$ and $\beta$ restricts to zero in $\SH^G(U)$, then $\alpha \wedge \beta =0$ in $\SH^G(X)$.
\end{lemma}
\begin{proof}
Let $i\colon Y \to X$ and $j\colon U \to X$ be the immersions. Write $\alpha \colon A \to A'$ and $\beta \colon B \to B'$. By \cite[Theorem~6.18 (4) (6) (7)]{Hoyois-equivariant} we have a commutative diagram in $\SH^G(X)$, where rows are distinguished triangles
\[ \xymatrix{
j_!j^*(A \wedge B)\ar[r] \ar[d]^{j_!j^*(\alpha \wedge \id_B)} & A \wedge B \ar[r] \ar[d]^{\alpha \wedge \id_B} & i_*i^*(A \wedge B) \ar[d]^{i_*i^*(\alpha \wedge \id_B)} \ar[r] & \Sigma^{1,0}j_!j^*(A \wedge B)\ar[d]^{\Sigma^{1,0}j_!j^*(\alpha \wedge \id_B)} \\ 
j_!j^*(A' \wedge B)\ar[r] \ar[d]^{j_!j^*(\id_{A'} \wedge \beta)} & A' \wedge B \ar[r] \ar[d]^{\id_{A'} \wedge \beta} & i_*i^*(A' \wedge B) \ar[d]^{i_*i^*(\id_{A'} \wedge \beta)} \ar[r] & \Sigma^{1,0}j_!j^*(A' \wedge B)\ar[d]^{\Sigma^{1,0}j_!j^*(\id_{A'} \wedge \beta)} \\
j_!j^*(A' \wedge B')\ar[r] & A' \wedge B' \ar[r] & i_*i^*(A' \wedge B') \ar[r] & \Sigma^{1,0}j_!j^*(A' \wedge B) \\
}\]
Since $i^*\alpha=0$ and $j^*\beta=0$, it follows from the monoidality of $i^*$ and $j^*$ that $i_*i^*(\alpha \wedge \id_B)=0$ and $j_!j^*(\id_{A'} \wedge \beta)=0$. A diagram chase then shows the vanishing of the composite $\alpha \wedge \beta = (\id_{A'} \wedge \beta) \circ (\alpha \wedge \id_B) \colon A \wedge B \to A'\wedge B'$.
\end{proof}

\begin{lemma}
\label{lemm:vanish_product_cover}
Let $X \in \Sch^G_S$. Let $U_1,\dots,U_n$ be a cover of $X$ by $G$-invariant open subschemes, and let $\alpha_1,\dots,\alpha_n$ be morphisms in $\SH^G(X)$. Assume that $\alpha_i$ restricts to zero in $\SH^G(U_i)$ for each $i \in \{1,\dots,n\}$. Then $\alpha_1 \wedge \cdots \wedge \alpha_n =0$ in $\SH^G(X)$.
\end{lemma}
\begin{proof}
We proceed by induction on $n$, the case $n=0$ being clear, as then $X=\varnothing$. Assume that $n \geq 1$. Let $Y$ be a $G$-invariant closed complement of $U_n$ in $X$ (such exists by \cite[Lemma~2.1]{Hoyois-equivariant}). Then $Y$ is covered by the $G$-invariant open subschemes $Y \cap U_1,\dots,Y \cap U_{n-1}$, and $\alpha_j$ restricts to zero in $\SH^G(Y\cap U_j)$ for $j=1,\dots,n-1$ (as $Y \cap U_j \subset U_j$). Thus by induction $\alpha_1 \wedge \dots \wedge \alpha_{n-1}$ restricts to zero in $\SH^G(Y)$. Since $\alpha_n$ restricts to zero in $\SH^G(U_n)$, we conclude by \rref{lemm:Euler_closed_open} that $\alpha_1 \wedge \dots \wedge \alpha_n$ restricts to zero in $\SH^G(X)$.
\end{proof}

\subsection{Motivic concentration}
We now fix a closed normal subgroup scheme $C \subset G$ over $S$. We assume that $C$ is linearly reductive, and that the $\Oc_S$-module $\Oc_S[C]$ is locally projective \rref{p:def_loc_proj}.

\begin{para}
Let $X \in \Sch^G_S$. Recall from \cite[Definition~2.18]{Hoyois-equivariant} that a $G$-equivariant morphism $Y \to X$ is called a \emph{$G$-affine bundle} if it is a torsor under a $G$-equivariant a vector bundle $E \to X$, in such a way that the action $E \times_XY \to Y$ is $G$-equivariant. From the description given just below \cite[Definition~2.18]{Hoyois-equivariant}, it follows that $Y$ is a $G$-invariant closed subscheme of a $G$-equivariant vector bundle over $X$, and in particular $Y \in \Sch^G_S$ (see \cite[Lemma~2.13]{Hoyois-equivariant}).
\end{para}

\begin{definition}
\label{def:V_C}
We set
\[
\mathcal{V}_C=\{\text{$G$-equivariant vector bundles $V\to S$ such that $V^C=0$}\}.
\]
\end{definition}

\begin{para}
\label{p:VC_stable_sum}
The set $\mathcal{V}_C$ is stable under direct sums.
\end{para}

\begin{proposition}
\label{cor:no_fixed_SH_vanishes}
Let $X \in \Sch^G_S$ be such that $X^C=\varnothing$. Then there exists $V \in \mathcal{V}_C$ such that $e_V =0$ in $\SH^G(X)$.
\end{proposition}
\begin{proof}
Let $U_1,\dots,U_n$ be a finite cover of $S$ by affine open subschemes. Let $k \in \{1,\dots,n\}$. By Jouanolou's trick \cite[Proposition~2.20]{Hoyois-equivariant} there exists a $G$-affine bundle $Y_k \to X \times_S U_k$ such that $Y_k$ is affine. Note that $Y_k^C=\varnothing$. By \rref{th:conc_main} we find an element $V_k \in \Vc_C$ whose pullback to $Y_k$ admits a nowhere vanishing $G$-equivariant section. By \rref{lemm:e_section} this implies that $e_{V_k} =0$ in $\SH^G(Y_k)$, hence by homotopy invariance \cite[Theorem~6.18~(8)]{Hoyois-equivariant} we deduce that $e_{V_k} =0$ in $\SH^G(X \times_S U_k)$. We conclude using \rref{lemm:vanish_product_cover} that $e_{V_1} \wedge \dots \wedge e_{V_n}$ vanishes in $\SH^G(X)$. Since $e_{V_1} \wedge \dots \wedge e_{V_n}$ may be identified with $e_{V_1 \oplus \dots\oplus V_n}$, we conclude by setting $V = V_1 \oplus \dots\oplus V_n$.
\end{proof}

\begin{definition}
\label{def:e_V_isom}
Let $V \to S$ be a $G$-equivariant vector bundle, and $X \in \Sch^G_S$. We will say that a morphism $f \colon A \to B$ in $\SH^G(X)$, resp.\ $\Hom(\Cc,\SH^G(X))$ for a category $\Cc$, is an \emph{$e_V$-isomorphism} if there exists a morphism $s \colon B \to \Sigma^VA$ fitting into a commutative diagram in $\SH^G(X)$, resp.\ $\Hom(\Cc,\SH^G(X))$,
\[ \xymatrix{
A\ar[rr]^-f \ar[d]_{e_{V,A}} && B \ar[d]^{e_{V,B}} \ar[dll]_s \\ 
\Sigma^VA \ar[rr]_-{\Sigma^V f} && \Sigma^VB 
}\]
\end{definition}

\begin{lemma}
\label{lemm:cone_e_V_isom}
Let $V \to S$ be a $G$-equivariant vector bundle, and $X \in \Sch^G_S$. Let $A \to B \to D \to \Sigma^{1,0}A$ be a distinguished triangle in $\SH^G(X)$.
\begin{enumerate}[(i)]
\item If $e_{V,D} \colon D \to \Sigma^V D$ is zero, then $A \to B$ is an $e_{V^{\oplus 2}}$-isomorphism.

\item If $A \to B$ is an $e_V$-isomorphism, then $e_{V^{\oplus 2},D} \colon D \to \Sigma^{V^{\oplus 2}} D$ is zero.
\end{enumerate}
\end{lemma}
\begin{proof}
In view of \rref{p:Euler_sum}, this follows from a diagram chase in the commutative diagram in $\SH^G(X)$, where rows are distinguished triangles,
\[
\begin{gathered}[b]
\xymatrix{
\Sigma^{-1,0}D\ar[r] \ar[d] & A \ar[r] \ar[d] & B \ar[r] \ar[d] & D \ar[d]\\ 
\Sigma^{-1,0}\Sigma^VD\ar[r] \ar[d] & \Sigma^VA \ar[r] \ar[d] & \Sigma^VB \ar[r]\ar[d] & \Sigma^VD \ar[d]\\
\Sigma^{-1,0}\Sigma^V\Sigma^VD\ar[r] & \Sigma^V\Sigma^VA \ar[r]   & \Sigma^V\Sigma^VB \ar[r] & \Sigma^V\Sigma^VD 
}\\[-\dp\strutbox]
\end{gathered}.
\qedhere
\]
\end{proof}

\begin{proposition}
\label{prop:euler_section}
Let $X \in \Sch^G_S$, and denote by $i \colon X^C \to X$ the closed immersion. Then there exists $V \in \mathcal{V}_C$ such that $\id_{\SH^G(X)} \to i_* i^*$ and $i_! i^! \to \id _{\SH^G(X)}$ are $e_V$-isomorphisms, as morphisms of functors $\SH^G(X) \to \SH^G(X)$.
\end{proposition}
\begin{proof}
Let $j \colon U \to X$ be the open complement of $i$. By \rref{cor:no_fixed_SH_vanishes}, there exists $W \in \mathcal{V}_C$ such that $e_W \colon \Un_U \to \Sigma^W\Un_U$ vanishes in $\SH^G(U)$. This morphism is the image of $e_W \colon \Un_X \to \Sigma^W\Un_X$ under $j^! \simeq j^* \colon \SH^G(X) \to \SH^G(U)$. Therefore $j_!j^!e_W$ and $j_*j^*e_W$ vanish in $\SH^G(X)$. By \rref{lemm:cone_e_V_isom}, the distinguished triangles in $\SH^G(X)$ (see \cite[Theorem~6.18~(4)]{Hoyois-equivariant})
\[
j_!j^!\Un_X \to \Un_X \to i_* i^* \Un_X \quad \text{and} \quad i_!i^!\Un_X \to \Un_X \to j_* j^* \Un_X
\]
thus imply that $\Un_X \to i_* i^* \Un_X$ and $i_!i^!\Un_X \to \Un_X$ are $e_V$-isomorphisms in $\SH^G(X)$, where $V=W^{\oplus 2} \in \Vc_C$. In view of the projection formula \cite[Theorem~6.18~(7)]{Hoyois-equivariant} (recall from \cite[Theorem~6.18 (1)]{Hoyois-equivariant} that $i_! \simeq i_*$), the proposition follows by smashing with $\id_{\SH^G(X)}$.
\end{proof}

\subsection{Inverting Euler classes}
We still assume $C \subset G$ is a closed normal subgroup scheme, which is linearly reductive, and that the $\Oc_S$-module $\Oc_S[C]$ is locally projective.

\begin{para}
We denote by $\Eu_C$ the collection of morphisms in $\SH^G(X)$ 
\[
A \xrightarrow{e_{V,A}} \Sigma^V A \xrightarrow{\sim} B,
\]
where $A,B \in \SH^G(X)$, and $V$ runs over $\Vc_C$. Using \rref{p:VC_stable_sum} and \rref{p:Euler_sum}, it is not difficult to show that $\Eu_C$ is a multiplicative system compatible with the triangulated structure in the sense of \cite[Tags~\href{https://stacks.math.columbia.edu/tag/04VC}{04VC}, \href{https://stacks.math.columbia.edu/tag/05R2}{05R2}]{stacks}. We will denote by $\SH^G(X)[\Eu_C^{-1}]$ the localization of the triangulated category $\SH^G(X)$ at the multiplicative system $\Eu_C$.
\end{para}

\begin{para}
\label{p:e_V_isom_in_loc}
If $f$ is an $e_V$-isomorphism in $\SH^G(X)$ for some $V \in \Vc_C$, then $f$ becomes an isomorphism in $\SH^G(X)[\Eu_C^{-1}]$.
\end{para}

\begin{para}
When $f \colon Y \to X$ is a morphism in $\Sch_S^G$, it follows from \cite[Theorem~6.18 (6) and (7)]{Hoyois-equivariant} that we have induced functors
\[
f^* \colon \SH^G(X)[\Eu_C^{-1}] \to \SH^G(Y)[\Eu_C^{-1}], \quad f_! \colon \SH^G(Y)[\Eu_C^{-1}] \to \SH^G(X)[\Eu_C^{-1}].
\]
\end{para}

\begin{theorem}
\label{th:localisation}
Let $X \in \Sch^G_S$ and denote by $i \colon X^C \to X$ the closed immersion. Then each of the adjunctions $i^* \dashv i_*$ and $i_! \dashv i^!$ induces an adjoint equivalence of categories
\[
\SH^G(X)[\Eu_C^{-1}] \simeq \SH^G(X^C)[\Eu_C^{-1}].
\]
\end{theorem}
\begin{proof}
Consider the cartesian square in $\Sch^G_S$
\[ 
\xymatrix{
X^C\ar[r]^{\id} \ar[d]_{\id} & X^C \ar[d]^i \\ 
X^C \ar[r]^i & X
}
\]
Then the exchange equivalence $Ex^*_*$, resp.\ $Ex^!_*$, (see \cite[Theorem~6.10, resp.\ Theorem~6.18~(3)]{Hoyois-equivariant}) shows that the counit $i^* \circ i_* \to \id$, resp.\ the unit $\id \to i^! \circ i_*=i^! \circ i_! $ is an isomorphism of functors $\SH^G(X^C) \to \SH^G(X^C)$. 

The fact that the unit $\id \to i_* \circ i^*$, resp.\ the counit $i_! \circ i^! \to \id$, is an isomorphism of functors $\SH^G(X)[\Eu_C^{-1}] \to \SH^G(X)[\Eu_C^{-1}]$ follows from \rref{prop:euler_section} and \rref{p:e_V_isom_in_loc}.
\end{proof}

\begin{para}
When $f \colon Y \to X$ is a smooth morphism in $\Sch_S^G$, it follows from the smooth projection formula \cite[Proposition~4.3]{Hoyois-equivariant} that we have an induced functor
\[
f_\sharp \colon \SH^G(Y)[\Eu_C^{-1}] \to \SH^G(X)[\Eu_C^{-1}].
\]
\end{para}

\begin{para}
\label{p:Gysin}
Let $i \colon Y \to X$ be a closed immersion in $\Sm^G_S$ with normal bundle $N \to Y$, and set $U=X\smallsetminus Y$. We consider the composite in $\SH^G(S)$
\[
\overline{i} \colon \Sup X \to \Su(X/U) \simeq \Su \Th N,
\]
where the last isomorphism is the purity isomorphism \cite[\S3.5]{Hoyois-equivariant}.
\end{para}

\begin{para}
\label{p:push_pull}
Let $i \colon Y \to X$ be a closed immersion in $\Sm^G_S$. It follows from the construction of the purity isomorphism \cite[\S3.5]{Hoyois-equivariant} that the composite $\overline{i} \circ \Sup i$ in $\SH^G(S)$ is the image of the morphism $e_N \colon \Un_Y \to \Sph^N$ under the map $(\pi_Y)_\sharp \colon \SH^G(Y) \to \SH^G(S)$.
\end{para}

\begin{corollary}
\label{cor:push_pull}
Let $X \in \Sm^G_S$. Denote by $i \colon X^C \to X$ the closed immersion, and by $N$ its normal bundle. Then we have a commutative diagram of isomorphisms in $\SH^G(S)[\Eu_C^{-1}]$
\[ \xymatrix{
\Sup X\ar[rr]^{\overline{i}}_{\sim} && \Su \Th N \\ 
& \Sup (X^C) \ar[lu]^{\Sup i}_{\sim} \ar[ru]_*+<0.5em>{_{(\pi_{X^C})_\sharp(e_N)}}^-{\sim} & 
}\]
\end{corollary}
\begin{proof}
The commutativity of the triangle was observed in \rref{p:push_pull}. Let $U=X \smallsetminus X^C$. Since $\Sup U=0 \in \SH^G(S)[\Eu_C^{-1}]$ by \rref{cor:no_fixed_SH_vanishes}, it follows that the morphism $\Sup X \to \Su(X/U)$ is an isomorphism in $\SH^G(S)[\Eu_C^{-1}]$, hence so is $\overline{i}$. Since $N \smallsetminus N^C=N^{\circ}$ by \rref{p:smooth_fixed_locus}, we have $\Sup (N^{\circ})=0 \in \SH^G(S)[\Eu_C^{-1}]$ by \rref{cor:no_fixed_SH_vanishes}, hence the morphism $\Sup N \to \Su(N/N^\circ)=\Th N$ is an isomorphism in $\SH^G(S)[\Eu_C^{-1}]$. So is the morphism $\Sup (X^C) \to \Sup N$ induced by the zero-section (by homotopy invariance \cite[Theorem~6.18~(8)]{Hoyois-equivariant}), hence so is their composite $\Sup (X^C) \to \Sup N \to \Th N$, which by \rref{def:Euler} coincides with $(\pi_{X^C})_\sharp(e_N)$.
\end{proof}

\begin{remark}
\label{rem:largest_quotient}
Observe that $\SH^G(S)[\Eu_C^{-1}]$ is the largest quotient of $\SH^G(S)$ where $\Sup (X^C) \to \Sup X$ induces an isomorphism for all $X \in \Sm^G_S$. Indeed assume that $F \colon \SH^G(S) \to \mathcal{C}$ is a tensor triangulated functor such that $F(\Sup (X^C)) \to F(\Sup X)$ is an isomorphism for any $X \in \Sm^G_S$. Let $V \in \mathcal{V}_C$. Taking $X=V^{\circ}$, we have $X^C =\varnothing$, and therefore $F(\Sup V^\circ)=0$ in $\mathcal{C}$. In view of the distinguished triangle of \rref{p:triangle_euler}, we deduce that $F(e_V)$ must be an isomorphism, which by monoidality implies that $F(e_{V,A})$ is an isomorphism in $\mathcal{C}$ for every $A \in \SH^G(S)$.
\end{remark}

\section{Cohomology theories}
\label{sect:coh}
We let $G$ be a group scheme over $S$.

\subsection{Equivariant theories}
\label{sect:equiv_theories}
In this section, we fix a spectrum $A \in \SH^G(S)$.

\begin{para}
When $V$ is a $G$-equivariant vector bundle over $X \in \Sch^G_S$, the object $\Sph^V=\Sigma^V\Un_X \in \SH^G(X)$ is invertible, see \cite[Proposition~6.5]{Hoyois-equivariant}. Moreover an exact sequence of $G$-equivariant vector bundles $0 \to V_1 \to V_2 \to V_3 \to 0$ over $X$ induces a canonical isomorphism $\Sph^{V_2} \simeq \Sph^{V_1} \wedge \Sph^{V_3}$, see \cite[p.348 and \S3.5]{Hoyois-equivariant}. As in \cite[\S4.1]{Riou}  this allows us to define an invertible object $\Sph^{\xi} \in \SH^G(X)$ for every virtual $G$-equivariant bundle $\xi$ over $X$, in the sense of Deligne \cite[\S4]{Deligne-det}. (See also \cite[\S16.2]{Bachmann-Hoyois-Norms}, and \cite[Théorème~15.18]{Ayoub-these-I} for details.)
\end{para}

\begin{para}
\label{p:def_coh}
Let $\xi$ be a virtual $G$-equivariant vector bundle over $X \in \Sch^G_S$. We set, for $p,q \in \Zz$,
\[
A^{p,q}(X;\xi) = [\Sigma^{-p,-q}\Sph^{-\xi}, \pi_X^*A]_{\SH^G(X)}, \quad A^{*,*}(X;\xi) = \bigoplus_{p,q\in\Zz} A^{p,q}(X;\xi),
\]
\[
A_{p,q}(X;\xi) = [\Sigma^{p,q}\Sph^{\xi}, \pi_X^!A]_{\SH^G(X)}, \quad A_{*,*}(X;\xi) = \bigoplus_{p,q\in\Zz} A_{p,q}(X;\xi).
\]
We will write $A^{p,q}(X)$, $A^{*,*}(X)$, $A_{p,q}(X)$, $A_{*,*}(X)$ instead of  $A^{p,q}(X;0)$, $A^{*,*}(X;0)$, $A_{p,q}(X;0)$, $A_{*,*}(X;0)$.
\end{para}

\begin{remark}
\label{rem:conv_twist}
When $G$ is trivial, the convention adopted in \rref{p:def_coh} agrees with that of \cite[Definition~2.2.1]{DFK}, \cite[\S3.1]{Levine-Atiyah-Bott} and \cite[2.1.8]{EHKSY}, but differs from that of \cite[\S5]{Ana-Pushforwards} (where $\xi$ is replaced with $-\xi$, and the grading is shifted by the rank of $\xi$).
\end{remark}

\begin{para}
If $\alpha \colon \xi \to \rho$ is an isomorphism of virtual $G$-equivariant vector bundles over $X \in \Sch^G_S$, then precomposing with the induced morphism $\Sph^\xi \to \Sph^\rho$ induces an isomorphism
\begin{equation}
\label{eq:pb_vb}
\alpha^* \colon A^{*,*}(X;-\rho) \to A^{*,*}(X;-\xi).
\end{equation}
\end{para}

\begin{para}
\label{p:action}
Assume that $A \in \SH^G(S)$ is a ring spectrum, by which we mean a monoid in $(\SH^G(S),\wedge,\Un_S)$, with product $\mu \colon A \wedge A \to A$. Then for any $X \in \Sch^G_S$, the group $A^{*,*}(X)$ is naturally a ring. More generally, when $\xi,\rho$ are virtual $G$-equivariant vector bundles over $X$, we have a product
\begin{equation}
\label{eq:prod}
A^{*,*}(X;\xi) \otimes A^{*,*}(X;\rho) \to A^{*,*}(X;\xi + \rho).
\end{equation}
Every element $a \in A^{p,q}(S)$ induces morphisms
\[
l_a\colon A \simeq \Un_S \wedge A \xrightarrow{a \wedge \id_A} (\Sigma^{p,q}A) \wedge A\simeq\Sigma^{p,q}(A \wedge A)  \xrightarrow{\Sigma^{p,q}\mu} \Sigma^{p,q}A,
\]
\[
r_a\colon A \simeq A \wedge \Un_S \xrightarrow{\id_A \wedge a} A \wedge (\Sigma^{p,q}A) \simeq\Sigma^{p,q}(A \wedge A)  \xrightarrow{\Sigma^{p,q}\mu} \Sigma^{p,q}A.
\]
Postcomposition with $\pi_X^*l_a$ and $\pi_X^*r_a$, resp. $\pi_X^!l_a$ and $\pi_X^!r_a$, induces a structure of $A^{*,*}(S)$-bimodule on $A^{*,*}(X;\xi)$, resp.\ $A_{*,*}(X;\xi)$, for every virtual $G$-equivariant vector bundle $\xi$ over $X$.
\end{para}

\begin{para}
\label{p:funct_coh}
Let $f \colon Y \to X$ be a morphism in $\Sch^G_S$, and $\xi$ a virtual $G$-equivariant vector bundle over $X$. Then the morphism $f^* \colon \SH^G(X) \to \SH^G(Y)$ induces a morphism
\begin{equation}
\label{eq:def_pb}
f^* \colon A^{*,*}(X;\xi) \to  A^{*,*}(Y;f^*\xi).
\end{equation}

Assume now that the morphism $f$ is proper and let $B =\Sigma^{p,q} \Sph^\xi \in \SH^G(X)$, with $p,q\in \Zz$. Then we have an isomorphism of functors $f_! \simeq f_*$ by \cite[Theorem~6.18~(2)]{Hoyois-equivariant}, and the composite
\[
[f^*B,\pi_Y^!A]_{\SH^G(Y)} \simeq [B,f_*\pi_Y^!A]_{\SH^G(X)} \simeq  [B,f_! f^!\pi_X^!A]_{\SH^G(X)} \to [B,\pi_X^!A]_{\SH^G(X)},
\]
where the last arrow is induced by the unit $\id \to f_!f^!$, induces a morphism
\begin{equation}
\label{eq:def_pf}
f_* \colon A_{*,*}(Y;f^*\xi) \to A_{*,*}(X;\xi).
\end{equation}
When $A$ is a ring spectrum, the maps \rref{eq:def_pb} and \rref{eq:def_pf} are morphisms of $A^{*,*}(S)$-bimodules.
\end{para}

\begin{para}
\label{p:relative_pb}
Consider a cartesian square in $\Sch^G_S$
\[ \xymatrix{
Y\ar[r]^g \ar[d]_{\pi_Y} & X \ar[d]^{\pi_X} \\ 
T \ar[r]^f & S
}\]
We view $X$, resp.\ $Y$, as an object in $\Sch^G_S$, resp.\ $\Sch^G_T$. Let $\xi$ be a virtual $G$-equivariant vector bundle over $X$, and $p,q\in \Zz$. Then the composite
\[
[\Sigma^{-p,-q}\Sph^{-\xi},\pi_X^*A]_{\SH^G(X)} \xrightarrow{g^*} [g^*\Sigma^{-p,-q}\Sph^{-\xi},g^*\pi_X^*A]_{\SH^G(Y)} \simeq [\Sigma^{-p,-q}\Sph^{-g^*\xi},\pi_Y^*f^*A]_{\SH^G(Y)}
\]
induces a morphism
\begin{equation}
\label{eq:relative_pb_coh}
g^* \colon A^{*,*}(X;\xi) \to (f^*A)^{*,*}(Y;g^*\xi).
\end{equation}
On the other hand the composite
\[
[\Sigma^{p,q} \Sph^\xi,\pi_X^!A]_{\SH^G(X)} \xrightarrow{g^*} [g^*\Sigma^{p,q} \Sph^\xi,g^*\pi_X^!A]_{\SH^G(Y)} \to [\Sigma^{p,q} \Sph^{g^*\xi},\pi_Y^!f^*A]_{\SH^G(Y)}
\]
where the unlabeled arrow is induced by the exchange morphism $Ex^{*!} \colon g^*\pi_X^! \to \pi_Y^!f^*$ (see \cite[p.272]{Hoyois-equivariant}), induces a morphism
\begin{equation}
\label{eq:relative_pb}
g^* \colon A_{*,*}(X;\xi) \to (f^*A)_{*,*}(Y;g^*\xi).
\end{equation}
\end{para}

\begin{para}
\label{p:pb_identify}
The morphism \eqref{eq:relative_pb_coh}, whose target is identified with $A^{*,*}(Y;g^*\xi)$ upon viewing $Y$ in $\Sch^G_S$, coincides with the pullback defined in \eqref{eq:def_pb}.
\end{para}

\begin{para}
\label{p:smooth_coh}
Let $V$ be a $G$-equivariant vector bundle over $X \in \Sm^G_S$. Then for each $p,q \in \Nn$ we have a natural identification
\[
A^{p,q}(X;-V) = [\Su \Th V, \Sigma^{p,q} A]_{\SH^G(S)}.
\]
Let us observe that, under this identification (for $V=0$), the pullback $f^* \colon A^{*,*}(X) \to A^{*,*}(Y)$ along a morphism $f \colon Y \to X$ in $\Sm^G_S$ is given by precomposition with the morphism $\Sup f \colon \Sup Y \to \Sup X$.

When $i\colon Y \to X$ is a closed immersion in $\Sm^G_S$ with normal bundle $N$, precomposing with the morphism $\overline{i} \colon \Sup X \to \Su \Th N$ in $\SH^G(S)$ (see \rref{p:Gysin}) induces a pushforward
\begin{equation}
\label{eq:pf}
i_* \colon A^{*,*}(Y;-N) \to A^{*,*}(X).
\end{equation}
When $A$ is a ring spectrum, this is a morphism of $A^{*,*}(S)$-bimodules.
\end{para}

\begin{remark}
When $\rho$ is a virtual $G$-equivariant vector bundle over $Z \in \Sm^G_S$, we have isomorphisms $A_{p,q}(Z;\rho) \simeq A^{-p,-q}(Z;\Tan_Z-\rho)$ by \cite[Theorem~6.18~(2)]{Hoyois-equivariant}, and the pushforward \rref{eq:pf} corresponds to the pushforward \rref{p:funct_coh} (for $f=i$ and $\xi=\Tan_X$).
\end{remark}

\begin{remark}
\label{rem:notation_G_trivial}
We will use the notation of this section also in the non-equivariant setting, that is when $G=1$, and thus $\SH^G(S)=\SH(S)$.
\end{remark}

\subsection{Pseudo-orientations}
\begin{definition}
\label{def:beta:1}
Let $A \in \SH^G(S)$ be a ring spectrum. A \emph{$G$-equivariant pseudo-orientation} of $A$ is the data of elements
\[
\beta_E \in A^{*,*}(X;-E)
\]
for each $G$-equivariant vector bundle $E$ over $X \in \Sm^G_S$, satisfying the following conditions:
\begin{enumerate}[(i)]
\item
\label{def:beta:1:central}
the class $\beta_E$ is $A^{*,*}(X)$-central (with respect to the product \rref{p:action}),

\item
\label{def:beta:1:homogeneous}
if $E$ has constant rank, then $\beta_E$ is homogeneous,

\item if $\alpha \colon E \to F$ is an isomorphism of $G$-equivariant vector bundles over $X \in \Sm^G_S$, then $\alpha^*\beta_F = \beta_E$ (see \eqref{eq:pb_vb}),

\item
\label{def:beta:1:pb}
if $f \colon Y \to X$ is a morphism in $\Sm^G_S$ and $E \to X$ is a $G$-equivariant vector bundle, then $f^*\beta_E=\beta_{f^*E}$,

\item
\label{def:beta:1:direct_sum}
if $E,F$ are $G$-equivariant vector bundles over $X \in \Sm^G_S$, then $\beta_E \cdot \beta_F = \beta_{E \oplus F}$ under the product \eqref{eq:prod}.
\end{enumerate}
By a \emph{pseudo-orientation} of $A \in \SH(S)$, we will mean a $1$-equivariant pseudo-orientation of $A$.
\end{definition}

\begin{para}
\label{p:beta:2}
Given a $G$-equivariant pseudo-orientation of a ring spectrum $A \in \SH^G(S)$ and a $G$-equivariant vector bundle $E$ over $X \in \Sm^G_S$, we consider the morphism $e_E^* \colon A^{*,*}(X;-E) \to A^{*,*}(X)$ given by precomposition with $e_E \colon \Un_X \to \Sph^E$, and define a class
\[
b(E) = e_E^*(\beta_E) \in A^{*,*}(X).
\]
Note that $b(E)$ is a central element in the ring $A^{*,*}(X)$.
\end{para}

\begin{lemma}
\label{b:sum}
Let $A \in \SH^G(S)$ be a ring spectrum equipped with a $G$-equivariant pseudo-orientation. Then $b(E \oplus F) = b(E) \cdot b(F)$ for every $G$-equivariant vector bundles $E,F$ over $X \in \Sm^G_S$.
\end{lemma}
\begin{proof}
This follows from \rref{p:Euler_sum} and \dref{def:beta:1}{def:beta:1:direct_sum}.
\end{proof}

\begin{example}
\label{ex:GL_pseudo}
If $A \in \SH(S)$ is a $\GL$-oriented commutative ring spectrum \cite[Definition~2.1]{PPR-MGL}, one may take for $\beta_E$ the element corresponding to $1$ under the Thom isomorphism $A^{*,*}(X;-E) \simeq A^{*,*}(X)$. This defines a pseudo-orientation of $A$, for which $b(E)$ is the Euler class of $E$, usually denoted by $e(E)$.
\end{example}

\begin{remark}
A pseudo-orientation of a commutative ring spectrum $A \in \SH(S)$ is a $\GL$-orientation if and only if $\beta_1 = 1$ under the canonical identification $A^{2,1}(S;-1) \simeq A^{0,0}(S)$.
\end{remark}

\begin{example}
\label{ex:hyp}
Assume that $A \in \SH(S)$ is a hyperbolically oriented ring spectrum \cite[Definition~2.2.2]{hyp} (for instance $A$ could be an $\Sp$-oriented commutative ring spectrum). Then one may take for $\beta_E$ the Euler class $e(E) \in A^{*,*}(X;-E)$; this defines a pseudo-orientation of $A$, for which $b(E)$ is the top Pontryagin class $\pi(E)$ (see \cite[Definition~2.2.7]{hyp}, recalling from \rref{rem:conv_twist} that the conventions for the twisted cohomology are ``opposite'').
\end{example}

\begin{remark}
An important difference with $\GL$-orientations is that in a pseudo-oriented ring spectrum $A \in \SH(S)$, the cup-product with the class $\beta_E$ need not induce an isomorphism $A^{*,*}(X) \xrightarrow{\sim} A^{*,*}(X;-E)$. For one thing, setting $\beta_E=0$ for all $E$ defines a pseudo-orientation on any ring spectrum $A$. A more interesting example is provided by \rref{ex:hyp}, where even when the hyperbolic orientation of $A$ is induced by a $\GL$-orientation, the cup-product with $\beta_E$ is not an isomorphism (as the Euler class is typically not invertible).
\end{remark}

\subsection{Concentration for equivariant theories}
\begin{para}
\label{p:G_conditions} 
We now let $G$ be a group scheme over $S$, and a $C\subset G$ a closed normal subgroup scheme. As before, we assume that:
\begin{enumerate}[(a)]
\item $G$ and $C$ are both linearly reductive \rref{def:lin_red},

\item $G$ has the resolution property \rref{def:resolution},

\item the $\Oc_S$-module $\Oc_S[C]$ is locally projective \rref{p:def_loc_proj}.
\end{enumerate}
\end{para}

\begin{para}
\label{p:eu_C}
Let $A \in \SH^G(S)$ be a ring spectrum equipped with a $G$-equivariant pseudo-orientation. Let us set
\[
\eu_C =\{b(V), \text{ where } V \in \mathcal{V}_C \} \subset A^{*,*}(S).
\]
By \rref{p:VC_stable_sum} and \rref{b:sum}, this is a multiplicative subset of central elements in the ring $A^{*,*}(S)$. When $M$ is a left, resp.\ right, $A^{*,*}(S)$-module, we will denote by $M[\eu_C^{-1}]$ its localization at this subset. Note that by \dref{def:beta:1}{def:beta:1:pb} and \dref{def:beta:1}{def:beta:1:central}, when $X \in \Sm^G_S$ the elements of $\eu_C$ act centrally on the bimodule $A^{*,*}(X)$; in particular the left and right localizations $A^{*,*}(X)[\eu_C^{-1}]$ are canonically isomorphic, and form a ring.
\end{para}

\begin{proposition}
\label{prop:conc_equiv_coh}
Let $A \in \SH^G(S)$ be a ring spectrum equipped with a $G$-equivariant pseudo-orientation. Then for any $X \in \Sch^G_S$ and virtual $G$-equivariant vector bundle $\xi$ over $X$, the closed immersion $i \colon X^C \to X$ induces isomorphisms, upon left, resp.\ right, localization:
\[
i^* \colon A^{*,*}(X;\xi)[\eu_C^{-1}] \xrightarrow{\mathmakebox[2em]{\sim}} A^{*,*}(X^C;i^*\xi)[\eu_C^{-1}],
\]
\[
i_* \colon A_{*,*}(X^C;i^*\xi)[\eu_C^{-1}] \xrightarrow{\mathmakebox[2em]{\sim}} A_{*,*}(X;\xi)[\eu_C^{-1}].
\]
\end{proposition}
\begin{proof}
By \rref{prop:euler_section}, we may find $V\in \mathcal{V}_C$ and morphisms $u,v$ fitting into commutative diagrams
\[
\xymatrix{
\pi_X^*A \ar[r] \ar[d]_{\pi_X^*e_{V,A}} & i_*i^*\pi_X^*A \ar[d]^{i_*i^*\pi_X^*e_{V,A}}\ar[ld]_u&& i_!i^!\pi_X^!A \ar[r] \ar[d]_{i_!i^!\pi_X^!e_{V,A}}& \pi_X^!A \ar[d]^{\pi_X^!e_{V,A}}\ar[ld]_v\\
\pi_X^*\Sigma^V A \ar[r] & i_*i^*\pi_X^*\Sigma^V A&&  i_!i^!\pi_X^!\Sigma^V A\ar[r]& \pi_X^!\Sigma^V A 
}
\]
Mapping out of $\Sigma^{-p,-q}\Sph^{-\xi}$, resp.\ $\Sigma^{p,q}\Sph^\xi$, we obtain commutative diagrams
\[
\xymatrix{
A^{*,*}(X;\xi)\ar[r]^-{i^*} \ar[d] & A^{*,*}(X^C;i^*\xi)\ar[d]\ar[ld] & A_{*,*}(X^C;i^*\xi) \ar[r]^-{i_*} \ar[d]& A_{*,*}(X;\xi) \ar[d]\ar[ld]\\
\Sigma^V A^{*,*}(X;\xi)\ar[r]^-{i^*} &\Sigma^V A^{*,*}(X^C;i^*\xi)& \Sigma^VA_{*,*}(X^C;i^*\xi) \ar[r]^-{i_*} & \Sigma^V A_{*,*}(X;\xi) 
}
\]
where the vertical arrows are induced by $e_{V,A}$ (postcomposing). Now, it follows from \rref{p:action} and \rref{p:beta:2} that the actions of $b(V) \in A^{*,*}(S)$ on $A^{*,*}(X;\xi)$ and $A_{*,*}(X^C;i^*\xi)$ factors through (precomposition by) the vertical arrows. (The case of the right actions uses the factorization of the morphism $A \simeq A \wedge \Un_S \xrightarrow{\id_A \wedge e_V} A \wedge \Sph^V$ as $A \xrightarrow{e_{V,A}} \Sigma^VA=\Sph^V \wedge A\simeq A \wedge \Sph^V$.) Since $b(V) \in \eu_C$, this implies the statements.
\end{proof}

\begin{lemma}
\label{lemm:summand}
Let $E$ be a $G$-equivariant vector bundle over $X \in \Sch^G_S$. Assume that $C$ acts trivially on $X$, and that $E^C=0$. Then there exists an element $V \in \Vc_C$ and a $G$-affine bundle $Y \to X$ with $Y$ affine over $S$, such that the pullback of $E$ along the composite $Y^C \subset Y \to X$ is a $G$-equivariant quotient of $\pi_{Y^C}^*V$.
\end{lemma}
\begin{proof}
By Jouanolou's trick \cite[Proposition~2.20]{Hoyois-equivariant} we may find a $G$-affine bundle $Y \to X$ such that $Y$ is affine over $S$. Let $f\colon Y^C \to Y \to X$ be the composite. Note that $(f^*E)^C=0$ by \rref{lemm:pullback_fixed}. Applying \rref{lemm:quotien_free} to the sheaf of sections of the $G$-equivariant vector bundle $f^*E \to Y^C$, we find $V \in \Vc_C$ and a surjective morphism $\pi_{Y^C}^*V \to f^*E$ of $G$-equivariant vector bundles over $Y^C$.
\end{proof}

\begin{lemma}
\label{lemm:is_invertible}
Let $A \in \SH^G(S)$ be a ring spectrum equipped with a $G$-equivariant pseudo-orientation. Let $E$ be a $G$-equivariant vector bundle over $X \in \Sm^G_S$. Assume that $C$ acts trivially on $X$, and that $E^C=0$. Then the element $b(E)$ becomes invertible in the ring $A^{*,*}(X)[\eu_C^{-1}]$.
\end{lemma}
\begin{proof}
We apply \rref{lemm:summand} and use its notation. Note that by \rref{p:smooth_fixed_locus} we have $Y^C \in \Sm^G_S$. The pullback $A^{*,*}(X) \to A^{*,*}(Y)$ is an isomorphism by homotopy invariance \cite[Theorem~6.18~(8)]{Hoyois-equivariant}, and by \rref{prop:conc_equiv_coh} the pullback $A^{*,*}(Y) \to A^{*,*}(Y^C)$ becomes an isomorphism after localizing at $\eu_C$. Therefore the pullback $A^{*,*}(X)[\eu_C^{-1}] \to A^{*,*}(Y^C)[\eu_C^{-1}]$ is a ring isomorphism. Thus, replacing $X$ with $Y^C$, we may assume that $X$ is affine over $S$, and that we have an exact sequence of $G$-equivariant vector bundles over $X$
\[
0 \to F \to \pi_X^*V \to E \to 0,
\]
for some $V \in \Vc_C$. Pick a finite cover of $S$ by affine open subschemes $U_1,\dots,U_n$. For each $k \in \{1,\dots,n\}$, the restriction to the affine scheme $X \times_S U_k$ of the above exact sequence splits in the category of $G$-equivariant vector bundles by \cite[Lemma~2.17]{Hoyois-equivariant}, in view of \rref{rem:lin_red}. Therefore by \rref{b:sum} the element $b(V) - b(F) b(E)$ restricts to zero in $A^{*,*}(X \times_S U_k)$. Then it follows from \rref{lemm:vanish_product_cover} that
\[
(b(V) - b(F) b(E))^n=0 \; \text{in $A^{*,*}(X)$}.
\]
In particular $b(E)$ divides $b(V)^n$ in $A^{*,*}(X)$, whence the result.
\end{proof}

\begin{proposition}
\label{prop:inverse_1}
Let $A \in \SH^G(S)$ be a ring spectrum equipped with a $G$-equivariant pseudo-orientation. Let $X \in \Sm^G_S$, and denote by $N$ the normal bundle to the closed immersion $i \colon X^C \to X$. Then $b(N)$ becomes invertible in $A^{*,*}(X^C)[\eu_C^{-1}]$ and the pushforward defined in \rref{eq:pf} verifies
\[
i_*(b(N)^{-1} \cdot \beta_N)=1 \in A^{*,*}(X)[\eu_C^{-1}].
\]
\end{proposition}
\begin{proof}
The first statement follows from \rref{p:smooth_fixed_locus} and \rref{lemm:is_invertible}. By \rref{p:push_pull} and \rref{p:smooth_coh} the composite
\[
A^{*,*}(X^C;-N) \xrightarrow{i_*} A^{*,*}(X) \xrightarrow{i^*} A^{*,*}(X^C)
\]
is given by precomposition with $e_N \colon \Un_{X^C} \to \Sph^N$ in $\SH^G(X^C)$, and thus maps $y \cdot \beta_N$ to $y \cdot b(N)$ for any $y \in A^{*,*}(X^C)$. Localizing at $\eu_C$, and taking $y=b(N)^{-1} \in  A^{*,*}(X^C)[\eu_C^{-1}]$, we deduce that $i^* \circ i_*(b(N)^{-1} \cdot \beta_N)=1$ in $A^{*,*}(X^C)[\eu_C^{-1}]$. Since $i^* \colon A^{*,*}(X)[\eu_C^{-1}] \to A^{*,*}(X^C)[\eu_C^{-1}]$ is a ring isomorphism by \rref{prop:conc_equiv_coh}, this implies that $i_*(b(N)^{-1} \cdot \beta_N)=1$.
\end{proof}

\begin{example}
Hoyois provided a definition of the $G$-equivariant homotopy $K$-theory of a scheme $X \in \Sch^G_S$, represented by a ring spectrum $\KGL_{[S/G]} \in \SH^G(S)$ such that $\KGL_{[X/G]} \simeq \pi_X^*\KGL_{[S/G]}$ (see \cite[\S5]{Hoyois-cdh-equiv}). When $V \to X$ is a $G$-equivariant vector bundle, the composite in $\SH^G(X)$
\[
\Sph^V \xrightarrow{\id \wedge 1} \Sph^V \wedge \KGL_{[X/G]} \simeq \KGL_{[X/G]},
\]
where the isomorphism is given by Bott periodicity (see \cite{Hoyois-cdh-equiv}), yields an element $\beta_V \in (\KGL_{[S/G]})^{0,0}(X;-V)$, endowing the ring spectrum $\KGL_{[S/G]} \in \SH^G(S)$ with a $G$-equivariant pseudo-orientation. Therefore \rref{prop:conc_equiv_coh} and \rref{prop:inverse_1} yield a concentration theorem for homotopy $K$-theory (involving $(\KGL_{[S/G]})^{*,*}$) and, when $S$ is regular, for equivariant $K$-theory of coherent sheaves (involving $(\KGL_{[S/G]})_{*,*}$, see \cite[Remark~5.7]{Hoyois-cdh-equiv}). This extends Thomason's concentration theorem for $K$-theory \cite{Thomason-Lefschetz} in two directions: it allows the group $G$ to be of multiplicative type instead of diagonalizable (in fact $G$ and $C$ need only satisfy the conditions of \rref{p:G_conditions}, but see \rref{rem:mult_type_oriented}), and includes the cohomological version for singular $X$ (using homotopy $K$-theory instead of $K$-theory).
\end{example}

\subsection{Borel's construction}
\label{sect:Borel}
In this section, we come back to assuming that $G$ is an arbitrary group scheme over $S$.

\begin{para}
\label{p:Schq}
We denote by $\Schq_S^G$ the full subcategory of $\Sch_S^G$ whose objects are those schemes $F$ with a free $G$-action such that the fppf quotient $F/G$ exists in $\Sch_S$. It follows from \cite[Proposition~7.1]{GIT} that when $Y \to F$ is a morphism in $\Sch_S^G$ and $F \in \Schq_S^G$, then $Y \in \Schq_S^G$. Given $F \in \Schq_S^G$ and $X \in \Sch^G_S$, we will write $X_F = (X \times_S F)/G \in \Sch_{F/G}$. A morphism $f \colon Y \to X$ in $\Sch_S^G$ induces a morphism $f_F \colon Y_F \to X_F$ in $\Sch_{F/G}$.
\end{para}

\begin{para}
\label{p:direct_system}
From now on, we will assume given a direct system $E_mG$ for $m \in \Nn$ in the category $\Schq^G_S$. For each $m$, we set $B_mG = E_m/G \in \Sch_S$, and we will additionally assume that $E_mG \in \Sm_S$, which implies that $B_mG \in \Sm_S$ by \cite[Tag~\href{https://stacks.math.columbia.edu/tag/05B5}{05B5}]{stacks}.
\end{para}

\begin{remark}
As suggested by the notation, the system $E_mG$ will typically satisfy additional conditions ensuring that the family $B_mG$ forms an approximation of $BG$ (see \rref{p:model} below), but such conditions will play no role in this section.
\end{remark}

\begin{para}
\label{p:supported}
Let $\Gamma$ be an abelian group and $M$ a $\Gamma$-graded group. For $\gamma \in \Gamma$, let us denote by $M^\gamma$ the graded component of $M$ of degree $\gamma$. For a subset $F \subset \Gamma$, we will say that an element $x \in M$ is \emph{supported in degrees $F$} if $x \in \bigoplus_{\gamma \in F} M^\gamma$. 

Let now $M_m$, for $m\in \Nn$, be an inverse system of $\Gamma$-graded groups. Its limit in the category of $\Gamma$-graded groups is the graded group $M=\lim_m M_m$ whose graded component $M^\gamma$, for $\gamma\in \Gamma$, is the limit of the (non-graded) abelian groups $(M_m)^\gamma$ over $m \in \Nn$. (Note that $M$ differs in general from the limit computed in the category of abelian groups.) In particular every element of $M$ is determined by its images in each $M_m$, for $m\in \Nn$. Conversely, a family $x_m \in M_m$ compatible with the transition maps defines an element of $M$ if and only if there exists a finite subset $F \subset \Gamma$ such that each $x_m$ is supported in degrees $F$.
\end{para}

\begin{definition}
\label{p:Borel}
(See \cite[\S1.2]{Di_Lorenzo-Mantovani}, \cite[\S4]{Levine-Atiyah-Bott}.) For every $X \in \Sch^G_S$ and $m\in \Nn$, we will write 
\[
X_m = X_{E_mG} = (X \times_S E_mG)/G \in \Sch_{B_mG}.
\]
If $V$ is a $G$-equivariant vector bundle over $X$, then $V_m$ is a $G$-equivariant vector bundle  on $X_m$. This extends naturally to virtual bundles: to a virtual $G$-equivariant vector bundle $\xi$ over $X$ one associates a virtual vector bundle $\xi_m$ over each $X_m$. Let now $A \in \SH(S)$, and write $A_m = (\pi_{B_mG})^*A \in \SH(B_mG)$. For $X$ and $\xi$ as above, we have for each $m \in \Nn$ natural morphisms (see \rref{p:relative_pb})
\[
(A_{m+1})^{*,*}(X_{m+1};\xi_{m+1}) \to (A_m)^{*,*}(X_m;\xi_m),
\]
\[
(A_{m+1})_{*,*}(X_{m+1};\xi_{m+1}) \to (A_m)_{*,*}(X_m;\xi_m).
\]
We set
\[
A_G^{*,*}(X;\xi) = \lim_m (A_m)^{*,*}(X_m;\xi_m) \quad \text{and} \quad A^G_{*,*}(X;\xi) = \lim_m (A_m)_{*,*}(X_m;\xi_m),
\]
the limits being computed in the category of $\Zz^2$-graded groups (see \rref{p:supported}).
\end{definition}

\begin{para}
If $A \in \SH(S)$ is a ring spectrum, then by \rref{p:action} the group $A^{*,*}_G(X)$ is naturally a ring for every $X \in \Sch^G_S$, and for any virtual $G$-equivariant vector bundle $\xi$ over $X$, the group $A^{*,*}_G(X;\xi)$, resp.\ $A^G_{*,*}(X;\xi)$, is an $A^{*,*}_G(S)$-bimodule.
\end{para}

\begin{para}
\label{p:module_over_nonequ}
Let $A \in \SH(S)$. Then the morphisms $\pi_{B_m G}^* \colon A^{*,*}(S) \to (A_m)^{*,*}(B_mG)$ (see \eqref{eq:relative_pb_coh}) are compatible with the transition maps as $m$ varies, hence induce a morphism
\[
A^{*,*}(S) \to A^{*,*}_G(S).
\]
When $A$ is a ring spectrum, this is a ring morphism.
\end{para}

\begin{para}
Let $A \in \SH(S)$. Let $f \colon Y \to X$ be a morphism in $\Sch^G_S$, and $\xi$ a virtual $G$-equivariant vector bundle over $X$. From \rref{p:funct_coh} we obtain a morphism $f^* \colon A^{*,*}_G(X;\xi) \to A^{*,*}_G(Y;f^*\xi)$, and if $f$ is proper, a morphism $f_*\colon A_{*,*}(Y;f^*\xi) \to A_{*,*}(X;\xi)$. If $A$ is a ring spectrum, the maps $f_*$ and $f^*$ are morphisms of $A^{*,*}_G(S)$-bimodules.
\end{para}

\begin{para}
\label{p:pf_Borel} 
Let $A \in \SH(S)$. Let $i\colon Y \to X$ be a closed immersion in $\Sm^G_S$ with normal bundle $N$. From \rref{eq:pf} we obtain a pushforward $i_*\colon A^{p,q}_G(Y;-N) \to A^{p,q}_G(X)$, which is a morphism of $A^{*,*}_G(S)$-bimodules when $A$ is a ring spectrum.
\end{para}

\begin{para}
\label{p:beta}
Let $A \in \SH(S)$ be a ring spectrum equipped with a pseudo-orientation. Let $V$ be a $G$-equivariant vector bundle over $X \in \Sm^G_S$. Then, in view of \rref{p:pb_identify}, the elements $\beta_{V_m} \in (A_m)^{*,*}(X_m;V_m)$, resp.\ $b(V_m) \in (A_m)^{*,*}(X_m)$, are compatible with the transition maps as $m$ varies.

We claim that these elements are supported in degrees $D$ (in the sense of \rref{p:supported}), where $D$ is a finite subset of $\Zz^2$ independent of $m$. To see this, we may assume that the vector bundle $V \to X$ has constant rank (because $X$ decomposes $G$-equivariantly as a disjoint union of finitely many open subschemes on which the rank of $V$ is constant). Then it follows from \dref{def:beta:1}{def:beta:1:homogeneous} that each $\beta_{V_m}$, resp.\ $b(V_m)$, is homogeneous of some degree $d_m \in \Zz^2$; the compatibility with the transition maps implies that we may arrange that $d_m$ does not depend on $m$. This proves the claim.

From the claim, it follows that these families define elements
\[
\beta^G_V \in A_G^{*,*}(X;-V) \quad \text{and} \quad b^G(V) \in A_G^{*,*}(X).
\]
\end{para}

\subsection{Concentration for Borel-type theories}
In this section, we work with groups schemes $C \subset G$ over $S$ verifying the conditions of \rref{p:G_conditions}. In addition, we assume given a system $E_mG$, for $m\in \Nn$ as in \rref{p:direct_system}. 

\begin{proposition}
\label{prop:euler_section:2}
Let $X \in \Sch^G_S$, and denote by $i\colon X^C \to X$ the closed immersion. Then there exists $V \in \mathcal{V}_C$ such that for any $F \in \Schq_S^G$, the morphisms $\id_{\SH(X_F)} \to (i_F)_* (i_F)^* \quad \text{and} \quad (i_F)_! (i_F)^! \to \id_{\SH(X_F)}$ are $e_{V_F}$-isomorphisms (see \rref{def:e_V_isom}).
\end{proposition}
\begin{proof}
Proceeding as in \rref{prop:euler_section}, we reduce first to assuming that  $X^C=\varnothing$, and finding $V \in \mathcal{V}_C$ such that $e_{V_F}=0$ for all $F \in \Schq_S^G$. Proceeding as in \rref{cor:no_fixed_SH_vanishes}, the statement follows from \rref{th:conc_main}, noting that if the vector bundle $\pi_X^*V \to X$ has a $G$-equivariant nowhere vanishing section, then $\pi_{X_F}^*(V_F) \to X_F$ has a nowhere vanishing section for all $F \in \Schq_S^G$. (Alternatively the statement may be deduced directly from \rref{prop:euler_section} using the functor $\SH^G(X) \to \SH^G(X \times_S F) \simeq \SH(X_F)$.)
\end{proof}

\begin{lemma}
\label{lemm:ker_coker}
Let $A \in \SH(S)$ be a ring spectrum equipped with a pseudo-orientation. Let $X \in \Sch^G_S$ and let $\xi$ be a virtual $G$-equivariant vector bundle over $X$. Then there exists $V \in \Vc_C$ such that, for each $m \in \Nn$, left, resp.\ right, multiplication by $b(V_m)$ annihilates the kernels and cokernels of the morphisms
\[
(A_m)^{*,*}(X_m;\xi_m) \xrightarrow{(i_m)^*} (A_m)^{*,*}((X^C)_m;i_m^*\xi_m)
\]
\[
(A_m)_{*,*}((X^C)_m;i_m^*\xi_m) \xrightarrow{(i_m)_*} (A_m)_{*,*}(X_m;\xi_m),
\]
where $i_m \colon (X^C)_m \to X_m$ denotes the closed immersion induced by $X^C \to X$.
\end{lemma}
\begin{proof}
Let us pick $V \in \Vc$ as in \rref{prop:euler_section:2}. The rest of the proof is similar to that of \rref{prop:conc_equiv_coh}. Namely, for each $m$, we may find morphisms $u_m,v_m$ fitting into commutative diagrams
\[
\xymatrix{
\pi_{X_m}^*A_m \ar[r] \ar[d]_{\pi_{X_m}^*e_{V_m}} & (i_m)_*i_m^*\pi_{X_m}^*A_m \ar[d]^{(i_m)_*i_m^*\pi_{X_m}^*e_{V_m}}\ar[ld]_{u_m}&& (i_m)_!i_m^!\pi_{X_m}^!A_m \ar[r] \ar[d]_{(i_m)_!i_m^!\pi_{X_m}^!e_{V_m}}& \pi_{X_m}^!A_m \ar[d]^{\pi_{X_m}^!e_{V_m}}\ar[ld]_{v_m}\\
\pi_{X_m}^*\Sigma^{V_m} A_m \ar[r] & (i_m)_*i_m^*\pi_{X_m}^*\Sigma^{V_m} A_m&& (i_m)_!i_m^!\pi_{X_m}^!\Sigma^{V_m} A_m\ar[r]& \pi_{X_m}^!\Sigma^{V_m} A_m 
}
\]
Mapping out of $\Sigma^{-p,-q}\Sph^{-\xi_m}$, resp.\ $\Sigma^{p,q}\Sph^{\xi_m}$, we obtain commutative diagrams, for every $m$,
\[
\xymatrix{
(A_m)^{*,*}(X_m;\xi_m)\ar[rr]^-{i_m^*} \ar[d] && (A_m)^{*,*}((X^C)_m;i_m^*\xi_m)\ar[d]\ar[lld]\\
(\Sigma^{V_m} A_m)^{*,*}(X_m;\xi_m)\ar[rr]_-{i_m^*} &&(\Sigma^{V_m} A_m)^{*,*}((X^C)_m;i_m^*\xi_m)
}
\]
\[
\xymatrix{
(A_m)_{*,*}((X^C)_m;i_m^*\xi_m) \ar[rr]^-{(i_m)_*} \ar[d]&& (A_m)_{*,*}(X_m;\xi_m) \ar[d]\ar[lld]\\
(\Sigma^{V_m}A_m)_{*,*}((X^C)_m;i_m^*{\xi_m}) \ar[rr]_-{(i_m)_*}&& (\Sigma^{V_m} A_m)_{*,*}(X_m;\xi_m) 
}
\]
where the vertical arrows are induced by $e_{V_m}$ (postcomposing). Now, it follows from \rref{p:action} and \rref{p:beta:2} that the actions of $b(V_m) \in A_m^{*,*}(B_mG)$ on the groups $A_m^{*,*}(X_m;\xi_m)$ and $(A_m)_{*,*}((X^C)_m;i_m^*\xi_m)$ factors through (precomposition by) the vertical arrows. This implies the statements.
\end{proof}

\begin{para}
\label{p:def_eu_C}
When $A \in \SH(S)$ is a ring spectrum equipped with a pseudo-orientation, we let $\eu_C \subset A_G^{*,*}(S)$ be the set of the elements $b^G(V)$, for $V \in \Vc_C$. As in \rref{p:eu_C}, we will denote by $M[\eu_C^{-1}]$ both the left and right localizations of an $A^{*,*}_G(S)$-bimodule $M$, noting that they coincide when $M=A^{*,*}_G(X)$ for $X \in \Sm^G_S$.
\end{para}

\begin{proposition}
\label{prop:conc_Borel-type}
Let $A \in \SH(S)$ be a ring spectrum equipped with a pseudo-orientation. Then for any virtual $G$-equivariant vector bundle $\xi$ over $X \in \Sch^G_S$, the closed immersion $i \colon X^C \to X$ induces isomorphisms, upon left, resp.\ right, localization:
\[
i^* \colon A^{*,*}_G(X;\xi)[\eu_C^{-1}] \xrightarrow{\mathmakebox[2em]{\sim}} A^{*,*}_G(X^C;i^*\xi)[\eu_C^{-1}]
\]
\[
i_* \colon A_{*,*}^G(X^C;i^*\xi)[\eu_C^{-1}] \xrightarrow{\mathmakebox[2em]{\sim}} A_{*,*}^G(X;\xi)[\eu_C^{-1}].
\]
\end{proposition}
\begin{proof}
This follows from \rref{lemm:ker_coker} and \rref{lemm:lim_loc} below.
\end{proof}

\begin{lemma}
\label{lemm:lim_loc}
Let $\Gamma$ be an abelian group. Let $R$ be a $\Gamma$-graded ring, and $A_m,B_m,C_m,D_m$ for $m \in \Nn$ inverse systems of $\Gamma$-graded left (resp.\ right) $R$-modules. Assume given exact sequences, compatible with the transition maps as $m$ varies,
\[
A_m \to B_m \to C_m \to D_m.
\]
Let $s\in R$ be a central element such that $sA_m=0$ and $sD_m=0$ (resp.\ $A_ms=0$ and $D_ms=0$) for all $m \in \Nn$. Then
\[
(\lim_m B_m)[s^{-1}] \to (\lim_m C_m)[s^{-1}]
\]
is an isomorphism.
\end{lemma}
\begin{proof}
We consider only the case of left modules, the other case being completely analogous. Let $x$ be an element of $\lim_m B_m$ whose image $y$ in $\lim_m C_m$ satisfies $s^ky =0$ for some $k \in \Nn$. Denote by $x_m \in B_m$ the image of $x$ for $m \in \Nn$. Then $s^kx_m$ has vanishing image in $C_m$, hence is the image of an element $A_m$. Since $sA_m=0$, we deduce that $s^{k+1}x_m=0$ for all $m\in \Nn$, and thus $s^{k+1}x=0$. This proves the injectivity of the morphism.

Let $c \in \lim_m C_m$, and denote by $c_m \in C_m$ its image for $m \in \Nn$. We may find a finite subset $F_c$, resp.\ $F_s$, of $\Gamma$ such that $c$, resp.\ $s$, is supported in degrees $F_c$, resp.\ $F_s$ (in the sense of \rref{p:supported}). Since $s$ annihilates the image of $c_m$ in $D_m$, the element $sc_m$ is the image of an element $b_m \in B_m$. As $sc_m$ is supported in degrees $F_s + F_c$, we may arrange that the same is true for $b_m$. Let $\varphi_{m+1} \colon B_{m+1} \to B_m$ be the transition map. The element $\varphi_{m+1}(b_{m+1}) - b_m \in B_m$ maps to zero in $C_m$ and is supported in degree $F_c+F_s$, hence is the image of an element $a_m \in A_m$ supported in degrees $F_c+F_s$. Since $sA_m=0$, we have $sa_m=0$, and we conclude that $\varphi_{m+1}(sb_{m+1}) = sb_m$. Moreover $sb_m \in B_m$ is supported in degrees $F_c +F_s+F_s$, a finite subset of $\Gamma$ which does not depend on $m$. Therefore the family $sb_m \in B_m$, for $m\in \Nn$, defines an element of $\lim_m B_m$, whose image in $\lim_m C_m$ is $s^2c$. This proves the surjectivity of the morphism.
\end{proof}

\begin{lemma}
\label{lemm:is_invertible:b}
Let $A \in \SH(S)$ be a ring spectrum equipped with a pseudo-orientation. Let $E$ be a $G$-equivariant vector bundle over $X \in \Sm^G_S$. Assume that $C$ acts trivially on $X$, and that $E^C=0$. Then the element $b^G(E)$ becomes invertible in the ring $A^{*,*}_G(X)[\eu_C^{-1}]$.
\end{lemma}
\begin{proof}
We apply \rref{lemm:summand} and use its notation. Note that by \rref{p:smooth_fixed_locus} we have $Y^C \in \Sm^G_S$. For each $m \in \Nn$ the morphism $Y_m \to X_m$ is an affine bundle, so that by homotopy invariance the pullback $(A_m)^{*,*}(X_m) \to (A_m)^{*,*}(Y_m)$ is an isomorphism, hence so is $A^{*,*}_G(X) \to A^{*,*}_G(Y)$. By \rref{prop:conc_equiv_coh} the pullback $A^{*,*}_G(Y) \to A^{*,*}_G(Y^C)$ becomes an isomorphism after localizing at $\eu_C$. Therefore the pullback $A^{*,*}_G(X)[\eu_C^{-1}] \to A^{*,*}_G(Y^C)[\eu_C^{-1}]$ is a ring isomorphism. Thus, replacing $X$ with $Y^C$, we may assume that $X$ is affine over $S$, and that we have an exact sequence of $G$-equivariant vector bundles over $X$
\[
0 \to F \to \pi_X^*V \to E \to 0,
\]
for some $V \in \Vc_C$. Pick a finite cover of $S$ by affine open subschemes $U^1,\dots,U^n$. For $k \in \{1,\dots,n\}$, the restriction to the affine scheme $X \times_S U^k$ of the above exact sequence splits in the category of $G$-equivariant vector bundles by \cite[Lemma~2.17]{Hoyois-equivariant}, in view of \rref{rem:lin_red}. Thus for every $m \in \Nn$ and $k \in \{1,\dots,n\}$, the vector bundles $\pi_{X_m}^*(V_m)$ and $F_m \oplus E_m$ have isomorphic restrictions to $(X \times_S U^k)_m$, so that by \rref{b:sum} the element $b(V_m) - b(F_m) b(E_m)$ restricts to zero in $(A_m)^{*,*}((X \times_S U^k)_m)$. Then it follows from \rref{lemm:vanish_product_cover} that
\[
\big(b(V_m) - b(F_m) b(E_m)\big)^n=0 \; \text{in $(A_m)^{*,*}(X_m)$}.
\]
In particular $b(V_m)^n$ is divisible by $b(E_m)$ for each $m \in \Nn$, where the integer $n$ is independent of $m$. Therefore $b(V_m)^n$ annihilates the kernel and cokernel of the morphism $(A_m)^{*,*}(X_m) \to (A_m)^{*,*}(X_m)$ given by multiplication by $b(E_m)$, for each $m \in \Nn$. Thus \rref{lemm:lim_loc} implies that multiplication by $b^G(E)$ induces an isomorphism $A^{*,*}_G(X)[\eu_C^{-1}] \xrightarrow{\sim} A^{*,*}_G(X)[\eu_C^{-1}]$, which means that $b^G(E)$ becomes invertible in the ring $A^{*,*}_G(X)[\eu_C^{-1}]$.
\end{proof}

\begin{proposition}
\label{prop:formula_Borel-type}
Let $A \in \SH(S)$ be a ring spectrum equipped with a pseudo-orientation. Let $X \in \Sm^G_S$, and denote by $N$ the normal bundle to the closed immersion $i \colon X^C \to X$. Then $b^G(N)$ becomes invertible in the ring $A^{*,*}_G(X^C)[\eu_C^{-1}]$, and the pushforward defined in \rref{p:pf_Borel} verifies
\[
i_*(b^G(N)^{-1} \cdot \beta^G_N)=1 \in A^{*,*}_G(X)[\eu_C^{-1}].
\]
\end{proposition}
\begin{proof}
The first statement follows from \rref{p:smooth_fixed_locus} and \rref{lemm:is_invertible:b}. By \rref{p:push_pull} and \rref{p:smooth_coh} the composite
\[
(A_m)^{*,*}((X^C)_m;-N_m) \xrightarrow{i_*} (A_m)^{*,*}(X_m) \xrightarrow{i^*} (A_m)^{*,*}((X^C)_m)
\]
is given by precomposition with $e_{N_m} \colon \Un_{(X^C)_m} \to \Sph^{N_m}$ in $\SH((X^C)_m)$, and thus maps $y \cdot \beta_{N_m}$ to $y \cdot b(N_m)$ for any $y \in (A_m)^{*,*}((X^C)_m)$. Passing to the limit, it follows that the composite $i^* \circ i_* \colon A^{*,*}_G(X^C;-N) \to A^{*,*}_G(X^C)$ maps $y \cdot \beta^G_N$ to $y \cdot b^G(N)$ for any $y \in A^{*,*}_G(X^C)$. Localizing at $\eu_C$, and taking $y=b^G(N)^{-1} \in  A^{*,*}_G(X^C)[\eu_C^{-1}]$, we deduce that $i^* \circ i_*(b^G(N)^{-1} \cdot \beta^G_N)=1$ in $A^{*,*}_G(X^C)[\eu_C^{-1}]$. Since $i^* \colon A^{*,*}_G(X)[\eu_C^{-1}] \to A^{*,*}_G(X^C)[\eu_C^{-1}]$ is a ring isomorphism by \rref{prop:conc_Borel-type}, this implies the formula of the proposition.
\end{proof}

\begin{example}
When $S$ is the spectrum of a field, taking for $A \in \SH(S)$ the motivic Eilenberg--Mac Lane spectrum representing motivic cohomology (which is $\GL$-oriented and therefore admits a pseudo-orientation by \rref{ex:GL_pseudo}), we obtain a concentration theorem for Borel--type equivariant motivic cohomology. This recovers Edidin--Graham's concentration theorem \cite{EG-Equ} for equivariant Chow ring under the actions of algebraic tori, and generalizes it (but see \rref{rem:mult_type_oriented} below).
\end{example}

\begin{remark}
\label{rem:eta}
In the situation of \rref{prop:conc_Borel-type}, assume additionally that $A$ is $\eta$-periodic (in the sense of \cite[(1.8)]{eta}, for instance $A$ could be the spectrum representing Witt groups, when $S$ is the spectrum of a field of characteristic not two). Assume that there exists a $G$-equivariant vector bundle $V \to S$ of odd rank such that $V^C=0$. This is the case for instance when $G$ is diagonalizable and nontrivial.

Then we claim that $A^{*,*}_G(S)[\eu_C^{-1}]=0$. Indeed, for every $m\in \Nn$, by \cite[(2.2.6)]{eta} the projection $(V_m)^{\circ} \to B_mG$ acquires a section in $\SH(B_mG)[\eta^{-1}]$, so that by \rref{lemm:e_section} we have $e_{V_m} =0$ in $\SH(B_mG)[\eta^{-1}]$. This implies that $b(V_m)=0 \in (A_m)^{*,*}(B_mG)$ for all $m$, and so $b^G(V)=0 \in A_G^{*,*}(S)$, proving the claim.
\end{remark}

\begin{proposition}
\label{rem:mult_type_oriented}
Assume that $S$ is the spectrum of a field $k$, and let $A \in \SH(S)$ be a $\GL$-oriented commutative ring spectrum \cite[Definition~2.1]{PPR-MGL}. If the group scheme $G$ is not of multiplicative type, then the groups $A^{*,*}_G(X)[\eu_C^{-1}]$ and $A_{*,*}^G(X)[\eu_C^{-1}]$ vanish for every $X \in \Sch^G_S$.
\end{proposition}
\begin{proof}
Let $\overline{k}$ be an algebraic closure of $k$. The $\overline{k}$-group scheme $G_{\overline{k}}$ is not diagonalizable, hence there exists a $\overline{k}[G]$-comodule $V'$, of dimension $r\geq 2$ over $\overline{k}$, which is simple (that is, admits no nontrivial subcomodule). Then there exists a subfield $\ell \subset \overline{k}$ such that $d=[\ell:k]$ is finite, and an $\ell[G]$-comodule $V$ such that $V'=V \otimes_\ell \overline{k}$. We may view $V$ as a $k[G]$-comodule (of dimension $rd$ over $k$), that we denote by $W$. We have $W \otimes_k \ell \simeq V \otimes_\ell (\ell \otimes_k \ell) \simeq V^{\oplus d}$ as $\ell[G]$-comodule. Thus if $L \subset W \otimes_k \overline{k} \simeq  V'^{\oplus d}$ is a $\overline{k}[G]$-subcomodule of dimension $1$ over $\overline{k}$, each of the $d$ projections $L \to V'$ is zero by simplicity of $V'$ (and the fact that $r>1$), hence $L=0$. This implies that $X= \Pp(W) \in \Sm^G_S$ satisfies $X^G=\varnothing$. By the concentration theorem \rref{prop:conc_Borel-type} we have
\begin{equation}
\label{eq:A_G_P_W}
A^{*,*}_G(\Pp(W))[\eu_C^{-1}] =0.
\end{equation}
But by the projective bundle theorem \cite[Theorem~3.9]{Panin-Oriented_I} we have for every $m \in \Nn$ an isomorphism of left $A^{*,*}(B_mG)$-modules
\[
A^{*,*}(B_mG)^{\oplus rd} \xrightarrow{\sim} A^{*,*}(\Pp(W_m)), \quad (x_0,\dots,x_{dr-1}) \mapsto \sum_{i=0}^{dr-1} x_i e(\Oc(1))^i,
\]
where $e$ denotes the Euler class. Those isomorphisms are compatible with the transition maps as $m$ varies, showing that the left $A^{*,*}_G(S)$-module $A^{*,*}_G(\Pp(W))$ is free of rank $rd$. Since $rd >0$, it follows from \rref{eq:A_G_P_W} that $A^{*,*}_G(S)[\eu_C^{-1}]=0$, which implies the statements.
\end{proof}

\begin{remark}
Instead of defining the Borel-type theories as we did in \rref{p:Borel}, one could take the limit at the level of spaces, setting for each $p,q \in \Zz$, when $X \in \Sm_S^G$,
\[
\widetilde{A}_G^{p,q}(X) = [\Sigma^{-p,-q} \Sup \colim_m X_m, A]_{\SH^G(S)}.
\]
The constructions are related by the Milnor exact sequence (see e.g.\ \cite[Lemma~2.1.3]{PMR}) 
\[ 
0 \to {\lim_m}^1 A^{p-1,q}(X_m) \to \widetilde{A}_G^{p,q}(X) \to A_G^{p,q}(X) \to 0.
\]
However, we do not see how to lift canonically the classes $b^G(V) \in A_G^{*,*}(X)$ of \rref{p:beta} to characteristic classes in $\widetilde{A}_G^{*,*}(X)$, which prevents us from stating a concentration theorem such as \rref{prop:conc_Borel-type} for the theory $\widetilde{A}_G$.
\end{remark}

\section{Smith theory following Dwyer--Wilkerson}
\label{sect:Smith}
\numberwithin{theorem}{section}
\numberwithin{lemma}{section}
\numberwithin{proposition}{section}
\numberwithin{corollary}{section}
\numberwithin{example}{section}
\numberwithin{notation}{section}
\numberwithin{definition}{section}
\numberwithin{remark}{section}

In this section we prove an algebraic version of a result of Dwyer--Wilkerson \cite{Dwyer-Wilkerson} in topology. We fix a prime number $p$, and let $S=\Spec k$ where $k$ is a perfect field of characteristic different from $p$.

\begin{para}
We will denote by $\Hh \in \SH(S)$ the motivic Eilenberg--Mac Lane spectrum $H\Fp$ representing motivic cohomology with $\Fp$-coefficients (see e.g.\ \cite[\S4.2-4.3]{Hoyois-cobordism}). This ring spectrum is equipped with a pseudo-orientation (in the sense of \rref{def:beta:1}), being commutative and $\GL$-oriented (see \rref{ex:GL_pseudo}). 
\end{para}

\begin{para}
\label{p:model}
Let $G$ be a group scheme over $S$ (as always $G$ is assumed to be flat, affine and of finite type). Consider a direct system of $G$-equivariant vector bundles $V_m \to S$, for $m \in \Nn$, and $G$-invariant open subschemes $U_m \subset V_m$ such that $U_m \in \Schq^G_S$ (see \rref{p:Schq}), and $U_m$ is mapped into $U_{m+1}$. Assume in addition that the codimension $c_m$ of the complement of $U_m$ in $V_m$ satisfies $c_m \geq m$. 

We will write $E_mG$ instead of $U_m$. Using the construction given in \rref{p:Borel}, this permits to define a $\Zz^2$-graded ring $\Hh^{*,*}_G(X)$ for every $X \in \Sch_S^G$, together with a morphism of $\Zz^2$-graded rings $\Hh^{*,*}(S)\to \Hh^{*,*}_G(S)$ by \rref{p:module_over_nonequ}. When $X \in \Sm_S^G$, as explained in \rref{p:beta}, every $G$-equivariant vector bundle $E \to X$ admits an Euler class $e(E) \in \Hh^{*,*}_G(X)$ (denoted $b^G(E)$ in \rref{p:beta}).

A standard argument (based on \cite[Lemma~3.5]{Vo-03}) shows that the ring $\Hh^{*,*}_G(X)$ as well as the classes $e(E)$ do not depend on the choice of the family $(U_m,V_m$).
\end{para}

\begin{para}
As explained in \cite[Remark~1.4]{Totaro-CHBG}, a family satisfying the conditions of \rref{p:model} always exists (recall that $S$ is the spectrum of a field).
\end{para}

\begin{para}
\label{p:model_prod}
Let $G,G'$ be group schemes over $S$, and $(U_m,V_m)$, resp.\ $(U_m',V_m')$, a family satisfying the conditions of \rref{p:model} for the group $G$, resp.\ $G'$. Then $(U_m \times_S U_m',V_m \oplus V_m')$ satisfies the conditions of \rref{p:model} for the group $G \times_S G'$.
\end{para}

\begin{para}
\label{p:forget_action}
Let $(a,b) \in \Zz^2$ and $X \in \Sch^G_S$. By \cite[Lemma~3.5]{Vo-03} and homotopy invariance, the composite $\Hh^{a,b}(X) \to \Hh^{a,b}(X \times_S V_m) \to \Hh^{a,b}(X \times_S U_m)$ is an isomorphism for $m >b$. This yields a forgetful morphism (which does not depend on the choice of $m>b$),
\[
\Hh_G^{a,b}(X) \to \Hh^{a,b}((X \times_S U_m)/G) \to \Hh^{a,b}(X \times_S U_m) \simeq \Hh^{a,b}(X).
\]
\end{para}

\begin{para}
\label{p:H_mup_S}
The group $\Hh^{*,*}_{\mu_p}(S)$ has been computed in \cite[Theorem~6.10, Corollary~6.2]{Vo-03}: there are elements $u \in \Hh^{1,1}_{\mu_p}(S)$ and $v \in \Hh^{2,1}_{\mu_p}(S)$ such that the left, resp.\ right, $\Hh^{*,*}(S)$-module $\Hh^{*,*}_{\mu_p}(S)$ is freely generated by the elements $v^i,v^iu$ for $i \in \Nn$. Let us record that $v$ is the Euler class of the $\mu_p$-equivariant line bundle corresponding to the canonical character of $\mu_p$.
\end{para}

\begin{para}
\label{p:H_mu_trivial}
It follows from \cite[Theorem~6.10]{Vo-03} that, for $X \in \Sm^{\mu_p}_S$ with trivial $\mu_p$-action, the morphism
\[
\Hh^{*,*}_{\mu_p}(S) \otimes_{\Hh^{*,*}(S)} \Hh^{*,*}(X) \to \Hh^{*,*}_{\mu_p}(X)
\]
is bijective.
\end{para}

\begin{definition}
We denote by $\Ac$ the free associative ring generated by symbols $\St^i$ for $i\in \Nn$ modulo the relation $\St^0=1$ (the Adem relations will play no role). By an $\Ac$-module, we will mean a left $\Ac$-module. When $N$ is an $\Ac$-module and $x\in N$, we will write $\St^i(x) \in N$ instead of $\St^i \cdot x$.

A ring $R$ which is an $\Ac$-module will be called an \emph{$\Ac$-algebra}\footnote{by which we mean an algebra over the Hopf algebra $\Ac$, for the coproduct corresponding to the formula \eqref{eq:Cartan}} if
\begin{equation}
\label{eq:Cartan}
\St^i(xy) = \sum_{r+s=i} \St^r(x) \St^s(y),\quad \text{for all $x,y \in R$.}
\end{equation}
A morphism of $\Ac$-algebras is a ring morphism which is compatible with the $\Ac$-module structures.
\end{definition}

\begin{para}
Let $A \to B$ and $A\to C$ be morphisms of $\Ac$-algebras. Then setting
\[
\St^i(b \otimes c) = \sum_{r+s=i} \St^r(b) \otimes \St^s(c),\quad \text{for $b\in B,c \in C$}
\]
endows the ring $B \otimes_A C$ with a structure of $\Ac$-algebra.
\end{para}

\begin{definition}
When $N$ is an $\Ac$-module, we define a subgroup $\Uns(N) \subset N$ as 
\[
\Uns(N) = \{x \in N \text{ for which there exists $j \in \Nn$ such that $\St^i(x)=0$ for all $i \geq j$}\}.
\]
An element of $N$ is called \emph{unstable} if it belongs to $\Uns(N)$, and we say that the $\Ac$-module $N$ is \emph{unstable} if $\Uns(N)=N$.
\end{definition}

\begin{para}
\label{p:Uns_subring}
When $R$ is an $\Ac$-algebra, the set $\Uns(R)$ is a subring of $R$.
\end{para}

\begin{para}
\label{p:unstable_tensor}
If $A \to B$ and $A\to C$ are morphisms of $\Ac$-algebras, with $B$ and $C$ unstable $\Ac$-modules, then the $\Ac$-module $B \otimes_A C$ is unstable.
\end{para}

\begin{para}
\label{p:Steenrod_operations}
When $X$ is a smooth $S$-scheme, the power operations constructed in \cite{Vo-03} yield an $\Ac$-algebra structure on $\Hh^{*,*}(X)$ (where the action of $\St^i$ corresponds to the operation denoted $\mathrm{P}^i$ in \cite{Vo-03}). When $X \in \Sm^{\mu_p}_S$, passing to the limit, we obtain an $\Ac$-algebra structure on $\Hh^{*,*}_{\mu_p}(X)$. It follows from \cite[Lemma~9.9]{Vo-03} that these $\Ac$-modules structures are unstable.
\end{para}

\begin{para}
\label{p:A_loc}
(See \cite[\S2]{Wilkerson}.) Let $R$ be an $\Ac$-algebra. Consider the \emph{total Steenrod operation}, defined as
\[
\St \colon R \to R[[t]], \quad x \mapsto \sum_{i \in \Nn} \St^i(x) t^i.
\]
(Here the variable $t$ is taken to be central.) This is a ring morphism, because of the formula \eqref{eq:Cartan}. Since $\St^0$ acts as the identity, if $S \subset R$ is a multiplicatively closed subset consisting of central elements, there is an induced ring morphism
\[
S^{-1}R \to (S^{-1}R)[[t]].
\]
Taking individual components (the $t^i$-coefficient for each $i$), this yields a structure of $\Ac$-algebra on $S^{-1}R$.
\end{para}

\begin{para}
It follows from \cite[Lemma~9.8, Lemma~9.9]{Vo-03} (after passing to the limit) that, in the $\Ac$-algebra $\Hh^{*,*}_{\mu_p}(S)$, we have
\[
\St^1(v)=v^p, \quad \St^i(v)=0 \text{ for $i>1$}, \quad \St^j(u)=0 \text{ for $j>0$}.
\]
\end{para}

\begin{lemma}
\label{lemm:P_v_-1}
In the $\Ac$-algebra $\Hh^{*,*}_{\mu_p}(S)[v^{-1}]$ (see \rref{p:A_loc}), we have
\[
\St^i(v^{-1}) = (-1)^i v^{i(p-1)-1} \text{ for $i \in \Nn$}.
\]
\end{lemma}
\begin{proof}
Using the notation of \rref{p:A_loc}, we have in $(\Hh^{*,*}_{\mu_p}(S)[v^{-1}])[[t]]$
\[
\St(v^{-1})=\St(v)^{-1} = (v+tv^p)^{-1} = v^{-1}(1+tv^{p-1})^{-1}=\sum_{i\in \Nn} (-1)^i v^{i(p-1)-1}t^i.\qedhere
\]
\end{proof}

\begin{lemma}
\label{lemm:unstable_mup}
Let $\Hh^{*,*}(S) \to R$ be a morphism of $\Ac$-algebras mapping $v$ to a central element in $R$. Then $v$ is a nonzerodivisor in $\Hh^{*,*}_{\mu_p}(S) \otimes_{\Hh^{*,*}(S)} R$. If the $\Ac$-module $R$ is unstable, then
\[
\Hh^{*,*}_{\mu_p}(S) \otimes_{\Hh^{*,*}(S)} R=\Uns((\Hh^{*,*}_{\mu_p}(S) \otimes_{\Hh^{*,*}(S)} R)[v^{-1}]).
\]
\end{lemma}
\begin{proof}
Set $P=\Hh^{*,*}_{\mu_p}(S) \otimes_{\Hh^{*,*}(S)} R$. Recall from \rref{p:H_mup_S} that the right $R$-module $P$ is free with basis $v^j,v^ju$:
\begin{equation}
\label{eq:decomp_R_N}
P = \Big(\bigoplus_{j \in \Nn} v^jR\Big) \oplus \Big(\bigoplus_{j \in \Nn} v^juR\Big).
\end{equation}
This implies the first statement, and permits to view $P$ as a subring of $P[v^{-1}]$.

Assume that $R$ is unstable. Then the $\Ac$-module $P$ is unstable by \rref{p:Steenrod_operations} and \rref{p:unstable_tensor}, so that $P \subset \Uns(P[v^{-1}])$. To conclude, we let $x\in P$ and $m \in \Nn$ be such that $y=v^{-m}x \in P[v^{-1}]$ is unstable, and we show that $y \in P$. Let us assume the contrary. Note that $vy$ is unstable (by \rref{p:Uns_subring}), and thus we may assume that $m$ is chosen minimal, so that $vy \in P$. By \rref{eq:decomp_R_N}, we may write $vy = vz + w +uw' $, with $w,w' \in R =\Hh^{*,*}(S) \otimes_{\Hh^{*,*}(S)} R  \subset P$ and $z \in P$. Then, for all $i \in \Nn$ we have in $P[v^{-1}]$
\begin{align*}
\St^i(y)  
&= \St^i(z) + \St^i(v^{-1}w) + \St^i(v^{-1}uw')  \\ 
&= \St^i(z) + \sum_{r+s=i} (-1)^rv^{r(p-1)-1} (\St^s(w) + u\St^s(w')),
\end{align*}
using \rref{lemm:P_v_-1}, and the fact that $\St^j(u)=0$ for $j>0$. Since both $y$ and $z$ are unstable, we have for $i$ large
\[
0=\sum_{r+s=i} (-1)^rv^{r(p-1)-1} (\St^s(w) + u\St^s(w')).
\]
Taking the $v^{i(p-1)-1}$-component, resp.\ $v^{i(p-1)-1}u$-component, in the decomposition \eqref{eq:decomp_R_N} and using the fact that $\St^0$ acts as the identity, we obtain that $w=0$, resp.\ $w'=0$. This implies that $y=z$ belongs to $P$, a contradiction.
\end{proof}

In the next statement, when $N$ is a  $\Zz^2$-graded abelian group and $a,b \in \Zz$, we denote by $N^{a,b}$ is $(a,b)$-th component of $N$.

\begin{lemma}
\label{lemm:lim_prod}
Let $M_m$, for $m\in \Nn$, be an inverse system of $\Zz^2$-graded left $\Hh^{*,*}(S)$-modules such that $(M_m)^{a,b}=0$ for $b<0$ and $m\in \Nn$. Then the morphism
\[
\Hh^{*,*}_{\mu_p}(S) \otimes_{\Hh^{*,*}(S)}(\lim_m M_m)  \to \lim_m (\Hh^{*,*}_{\mu_p}(S)\otimes_{\Hh^{*,*}(S)} M_m)
\]
is bijective, where the limits are computed in the category of $\Zz^2$-graded abelian groups.
\end{lemma}
\begin{proof}
Let $N$ be a $\Zz^2$-graded left $\Hh^{*,*}(S)$-module such that $N^{i,j}=0$ for $j<0$. This is in particular the case for $N=M_m$, or $N = \lim_m M_m$. Let $(a,b) \in \Zz^2$, and let us write
\[
N\{a,b\} = \Big(\prod_{s=0}^b N^{a-2s,b-s}\Big) \times \Big(\prod_{s=0}^{b-1} N^{a-2s-1,b-s-1}\Big).
\]
Then the morphism
\[
N\{a,b\} \to (\Hh^{*,*}_{\mu_p}(S) \otimes_{\Hh^{*,*}(S)} N)^{a,b}
\]
given by
\[
(x_0,\dots,x_b),(y_0,\dots,y_{b-1}) \mapsto \sum_{s=0}^b v^s \otimes x_s + \sum_{s=0}^{b-1} v^su \otimes y_s
\]
is bijective by the assumption on $N$ (and by \rref{p:H_mup_S}). Thus in the commutative diagram
\[ \xymatrix{
\Big(\Hh^{*,*}_{\mu_p}(S) \otimes_{\Hh^{*,*}(S)}(\lim_m M_m)\Big)^{a,b} \ar[r]  &  \Big(\lim_m (\Hh^{*,*}_{\mu_p}(S)\otimes_{\Hh^{*,*}(S)} M_m)\Big)^{a,b} \\
& \lim_m \Big((\Hh^{*,*}_{\mu_p}(S)\otimes_{\Hh^{*,*}(S)} M_m)^{a,b}\Big) \ar@{=}[u]\\
(\lim_m M_m)\{a,b\} \ar[uu]\ar[r] & \lim_m(M_m\{a,b\}) \ar[u]
}\]
(see \rref{p:supported} for the vertical equality), the vertical arrows are isomorphisms of abelian groups . So is the lower horizontal arrow, as limits commute with products. Therefore the upper horizontal arrow is an isomorphism of abelian groups, as required.
\end{proof}

\begin{lemma}
\label{lemm:Smith}
Let $D$ be a group scheme over $S$, and consider the group scheme $G=\mu_p \times_S D$. Let $X \in \Sm^G_S$, and assume that $\mu_p$ acts trivially on $X$. Then the morphism
\[
\Hh^{*,*}_{\mu_p}(S) \otimes_{\Hh^{*,*}(S)} \Hh^{*,*}_D(X) \to \Hh^{*,*}_G(X)
\]
is bijective.
\end{lemma}
\begin{proof}
Recall from \rref{p:model_prod} that we may take $E_m G = E_m \mu_p \times_S E_mD$. Then, computing the limits in the category of $\Zz^2$-graded abelian groups:
\begin{align}
\label{eq:HG_mu_p_D:1}
\begin{split}
\Hh^{*,*}_G(X) 
&= \lim_m \Hh^{*,*}((E_mG \times_S X)/G)  \\ 
&= \lim_m \Hh^{*,*}((E_m \mu_p)/\mu_p \times_S (E_m D \times_S X)/D)\\
&= \lim_m \lim_{m'} \Hh^{*,*}((E_{m'}\mu_p)/\mu_p \times_S (E_m D \times_S X)/D)\\
&= \lim_m \Hh^{*,*}_{\mu_p}((E_m D \times_S X)/D).
\end{split}
\end{align}
Note that, for each $m \in \Nn$, the $S$-scheme $(E_m D \times_S X)/D$ is smooth \cite[Tag~\href{https://stacks.math.columbia.edu/tag/05B5}{05B5}]{stacks}, hence by \rref{p:H_mu_trivial} the morphism
\[
\Hh^{*,*}_{\mu_p}(S) \otimes_{\Hh^{*,*}(S)} \Hh^{*,*}((E_m D \times_S X)/D) \to \Hh^{*,*}_{\mu_p}((E_m D \times_S X)/D)
\]
is an isomorphism. Passing to the limit, it follows that the morphism
\begin{equation}
\label{eq:HG_mu_p_D:3}
\lim_m \Hh^{*,*}_{\mu_p}(S) \otimes_{\Hh^{*,*}(S)} \Hh^{*,*}((E_m D \times_S X)/D) \to \lim_m \Hh^{*,*}_{\mu_p}((E_m D \times_S X)/D)
\end{equation}
is an isomorphism. So is the morphism
\begin{equation}
\label{eq:HG_mu_p_D:2}
\Hh^{*,*}_{\mu_p}(S) \otimes_{\Hh^{*,*}(S)} \Hh^{*,*}_D(X) \to \lim_m \Hh^{*,*}_{\mu_p}(S) \otimes_{\Hh^{*,*}(S)} \Hh^{*,*}((E_m D \times X)/D)
\end{equation}
by \rref{lemm:lim_prod}, in view of the vanishing of motivic cohomology in negative weights. We conclude by combining the isomorphisms \eqref{eq:HG_mu_p_D:2} and \eqref{eq:HG_mu_p_D:3} with the computation \eqref{eq:HG_mu_p_D:1}.
\end{proof}

\begin{lemma}
\label{lemm:Smith:2}
Let $D$ be a group scheme over $S$, and let $C=(\mu_p)^r$ for some $r \in \Nn$. Consider the group scheme $G=C \times_S D$. Let $X \in \Sm^G_S$, and assume that $C$ acts trivially on $X$. Then the morphism
\[
\Hh^{*,*}_C(S) \otimes_{\Hh^{*,*}(S)} \Hh^{*,*}_D(X) \to \Hh^{*,*}_G(X)
\]
is bijective.
\end{lemma}
\begin{proof}
We proceed by induction on $r$, the case $r=0$ being clear. Assume that $r>0$, and let $C'=(\mu_p)^{r-1}$. Consider the commutative diagram
\[ \xymatrix{
\Hh^{*,*}_{C'}(S) \otimes_{\Hh^{*,*}(S)} \Hh^{*,*}_{\mu_p}(S) \otimes_{\Hh^{*,*}(S)}\Hh^{*,*}_D(X)\ar[r] \ar[d] & \Hh^{*,*}_{C'}(S) \otimes_{\Hh^{*,*}(S)} \Hh^{*,*}_{\mu_p \times_S D}(X) \ar[d] \\ 
\Hh^{*,*}_C(S) \otimes_{\Hh^{*,*}(S)} \Hh^{*,*}_D(X) \ar[r] & \Hh^{*,*}_G(X)
}\]
The upper horizontal and the left vertical arrows are isomorphisms by \rref{lemm:Smith}. The right vertical arrow is an isomorphism by induction. Therefore the lower horizontal arrow is an isomorphism, as required.
\end{proof}

\begin{proposition}
\label{prop:Smith}
Let $G \simeq (\mu_p)^n$ for some $n \in \Nn$, and let $C$ be a subgroup scheme of $G$. Let $X \in \Sm^G_S$, and assume that $C$ acts trivially on $X$. Then the set $\eu_C \subset \Hh^{*,*}_G(S)$ (see \rref{p:def_eu_C}) consists of nonzerodivisors in $\Hh^{*,*}_G(X)$, and we have
\[
\Hh^{*,*}_G(X)=\Uns(\Hh^{*,*}_G(X)[\eu_C^{-1}]).
\]
\end{proposition}
\begin{proof}
Since $S$ is a field, each element of $\Vc_C$ (see \rref{def:V_C}) splits as a direct sum of $G$-equivariant line bundles $L$ over $S$ such that $L^C=0$. In addition any such $L$ is trivial as a line bundle over $S$.

Now, let $L$ be a $G$-equivariant line bundle over $S$ such that $L^C=0$. Then there exists an embedding $\mu_p \subset C$ such that the $\mu_p$-action on $L$ corresponds to the canonical character. We may find an $S$-group scheme $D$ and an isomorphism $G \simeq \mu_p \times_S D$ which is compatible with the embedding $\mu_p \subset C \subset G$. The morphism $\Hh^{*,*}_{\mu_p}(S) \otimes_{\Hh^{*,*}(S)} \Hh^{*,*}_D(X) \to \Hh^{*,*}_G(X)$ maps $v \otimes 1$ to $\pi_X^*e(L)$, and is an isomorphism by \rref{lemm:Smith}. It thus follows from \rref{lemm:unstable_mup} that the element $e(L)$ is a nonzerodivisor in $\Hh^{*,*}_G(X)$, and that we have
\begin{equation}
\label{eq:e(L)_unstable}
\Hh^{*,*}_G(X) = \Uns(\Hh^{*,*}_G(X)[e(L)^{-1}]).
\end{equation}
Since every element of $\eu_C$ is a product of classes $e(L)$ for $L$ as above, we obtain in particular the first statement.

Next let $y \in \Uns(\Hh^{*,*}_G(X)[\eu_C^{-1}])$. Then there exist $G$-equivariant line bundles $L_1,\dots,L_s$ over $S$ satisfying $L_i^C=0$ for $i=1,\dots,s$, and such that $e(L_1)\cdots e(L_s) y \in \Hh^{*,*}_G(X)$. We prove by descending induction on $r$ that $e(L_1) \cdots e(L_r) y \in \Hh^{*,*}_G(X)$. This is true if $r=s$, so let us assume that $r<s$. The element $y'=e(L_1) \cdots e(L_r) y \in \Hh^{*,*}_G(X)[\eu_C^{-1}]$ is unstable by \rref{p:Uns_subring} (and \rref{p:Steenrod_operations}), and we have $e(L_{r+1})y' \in \Hh^{*,*}_G(X)$ by induction. It follows from \rref{eq:e(L)_unstable} applied with $L=L_{r+1}$ that $y' \in \Hh^{*,*}_G(X)$, concluding the inductive proof. The case $r=0$ then yields the last statement of the proposition.
\end{proof}

In the next theorem, we use the forgetful morphism described in \rref{p:forget_action}.

\begin{theorem}
Let $G \simeq (\mu_p)^n$ for some $n \in \Nn$, and let $C$ be a subgroup scheme of $G$. Let $X \in \Sm^G_S$ and denote by $i \colon X^C \to X$ the closed immersion. Then the pullback $i^* \colon \Hh^{*,*}_G(X) \to \Hh^{*,*}_G(X^C)$ induces an isomorphism $\Uns(\Hh^{*,*}_G(X)[\eu_C^{-1}]) \simeq \Hh^{*,*}_G(X^C)$. Moreover we have an isomorphism
\[
\Hh^{*,*}_{G/C}(X^C) \simeq \Hh^{*,*}(S) \otimes_{\Hh^{*,*}_C(S)} \Uns(\Hh^{*,*}_G(X)[\eu_C^{-1}]).
\]
\end{theorem}
\begin{proof}
The group schemes $G$ and $C$ verify the conditions of \rref{p:G_conditions} (being diagonalizable). By the concentration theorem \rref{prop:conc_Borel-type}, the map $i^*$ induces an isomorphism $\Hh^{*,*}_G(X)[\eu_C^{-1}] \xrightarrow{\sim} \Hh^{*,*}_G(X^C)[\eu_C^{-1}]$, which is compatible with the $\Ac$-module structures. The first statement then follows from \rref{prop:Smith} (applied to the scheme $X^C$, which is smooth over $S$ by \rref{p:smooth_fixed_locus}). The inclusion $C \subset G$ extends to an isomorphism $G \simeq C \times_S D$ for some subgroup scheme $D \subset G$, which is isomorphic to $G/C$. As $C \simeq (\mu_p)^r$ for some $r \in \Nn$, it follows from \rref{lemm:Smith:2} that
\[
\Hh^{*,*}_D(X^C) \simeq \Hh^{*,*}(S) \otimes_{\Hh^{*,*}_C(S)} \Hh^{*,*}_C(S) \otimes_{\Hh^{*,*}(S)} \Hh^{*,*}_D(X^C) \simeq \Hh^{*,*}(S) \otimes_{\Hh^{*,*}_C(S)} \Hh^{*,*}_G(X^C),
\]
which yields the last statement.
\end{proof}

Taking $C=G$, we obtain:
\begin{corollary}
\label{cor:Smith}
Let $G=(\mu_p)^n$ for some $n \in \Nn$, and $X \in \Sm^G_S$. Then we have a natural isomorphism
\[
\Hh^{*,*}(X^G) \simeq \Hh^{*,*}(S) \otimes_{\Hh^{*,*}_G(S)} \Uns(\Hh^{*,*}_G(X)[\eu_G^{-1}]).
\]
\end{corollary}

\begin{example}
Let $G \simeq (\mu_p)^n$ for some $n \in \Nn$, and $X \in \Sm^G_S$. Assume that $X$ has the $G$-equivariant cohomology of a $G$-equivariant vector bundle, by which we mean that the pullback $\Hh^{*,*}_G(S) \to \Hh^{*,*}_G(X)$ is an isomorphism. Then we claim that $X^G$ has the $G$-equivariant cohomology of the point, that is $\Hh^{*,*}(S) \to \Hh^{*,*}(X^G)$ is an isomorphism; in particular $X^G$ is geometrically connected. Indeed by \rref{cor:Smith} we have $\Hh^{*,*}(X^G) \simeq \Hh^{*,*}(S) \otimes_{\Hh^{*,*}_G(S)} \Uns(\Hh^{*,*}_G(S)[\eu_G^{-1}])$, while taking $C=G$ in \rref{prop:Smith} we have $\Hh^{*,*}_G(S) = \Uns(\Hh^{*,*}_G(S)[\eu_C^{-1}])$.
\end{example}

\newcommand{\etalchar}[1]{$^{#1}$}
\def\cprime{$'$}

\end{document}